\documentclass[11pt]{amsart}

\usepackage[T1]{fontenc}
\usepackage{lmodern}
\usepackage{microtype}
\usepackage{amsmath,amssymb,amsthm,mathtools}
\usepackage[margin=1in]{geometry}
\usepackage{enumitem}
\usepackage[colorlinks=true,linkcolor=blue,citecolor=blue,urlcolor=blue]{hyperref}
\usepackage{aliascnt}

\newtheorem{theorem}{Theorem}[section]
\newaliascnt{proposition}{theorem}
\newtheorem{proposition}[proposition]{Proposition}
\aliascntresetthe{proposition}
\newaliascnt{lemma}{theorem}
\newtheorem{lemma}[lemma]{Lemma}
\aliascntresetthe{lemma}
\newaliascnt{corollary}{theorem}
\newtheorem{corollary}[corollary]{Corollary}
\aliascntresetthe{corollary}
\newaliascnt{question}{theorem}
\newtheorem{question}[question]{Question}
\aliascntresetthe{question}
\newaliascnt{definition}{theorem}
\newtheorem{definition}[definition]{Definition}
\aliascntresetthe{definition}
\newaliascnt{remark}{theorem}
\newtheorem{remark}[remark]{Remark}
\aliascntresetthe{remark}

\usepackage[nameinlink,capitalise]{cleveref}
\crefname{theorem}{Theorem}{Theorems}
\Crefname{theorem}{Theorem}{Theorems}
\crefname{proposition}{Proposition}{Propositions}
\Crefname{proposition}{Proposition}{Propositions}
\crefname{lemma}{Lemma}{Lemmas}
\Crefname{lemma}{Lemma}{Lemmas}
\crefname{corollary}{Corollary}{Corollaries}
\Crefname{corollary}{Corollary}{Corollaries}
\crefname{question}{Question}{Questions}
\Crefname{question}{Question}{Questions}
\crefname{definition}{Definition}{Definitions}
\Crefname{definition}{Definition}{Definitions}
\crefname{remark}{Remark}{Remarks}
\Crefname{remark}{Remark}{Remarks}

\numberwithin{equation}{section}

\newcommand{\E}{\mathbb E}
\newcommand{\Prob}{\mathbb P}

\newcommand{\girth}{\operatorname{girth}}
\newcommand{\chif}{\chi_{\mathrm f}}

\newcommand{\calM}{\mathcal M}

\newcommand{\eps}{\varepsilon}

\title[Polynomial chromatic-sparsity]{The Erd\H{o}s--Hajnal High-Girth Subgraph Conjecture Holds in the Polynomial Chromatic-Sparsity Regime}

\author[E. Li]{Eric Li}
\dedicatory{\normalfont\normalsize Trinity College, University of Cambridge}
\date{June 14, 2026}
\thanks{Email addresses: \href{mailto:el593@cam.ac.uk}{el593@cam.ac.uk}, \href{mailto:contact@ericli.com}{contact@ericli.com}.}


\hypersetup{
  pdftitle={The Erd\H{o}s--Hajnal High-Girth Subgraph Conjecture Holds in the Polynomial Chromatic-Sparsity Regime},
  pdfauthor={Eric Li},
  pdfsubject={The Erd\H{o}s--Hajnal high-girth subgraph conjecture in the polynomial chromatic-sparsity regime},
  pdfkeywords={chromatic number, girth, polynomial chromatic sparsity, fractional chromatic number, random subgraphs, Erd\H{o}s--Hajnal problem}
}

\begin{document}

\begin{abstract}
For a graph $G$ put
\[
        h_r(G)=\max\{\chi(H):H\subseteq G,\ \girth(H)\ge r\}.
\]
Erd\H{o}s and Hajnal asked whether $h_r(G)\to\infty$ as $\chi(G)\to\infty$, for every fixed $r\ge4$.  We prove this in every fixed polynomial edge-density regime: for all $r\ge4$, $k\ge2$, $P,C>0$, there is $M=M_{r,k}(P,C)$ such that
\[
        \chi(G)\ge M,
        \qquad
        e(G)\le C\chi(G)^P
        \quad\Longrightarrow\quad
        h_r(G)\ge k.
\]
Quantitatively, after replacing $P$ by $P\vee2$ and $C$ by $C\vee2$,
\[
        M_{r,k}(P,C)
        \le \exp\!\left(O_{r,k}\bigl((P+2+\log(C\vee2))^2\bigr)\right),
\]
and consequently the same conclusion holds throughout the quasi-polynomial range
\[
        e(G)\le \exp\bigl(C_0(\log\chi(G))^a\bigr),
        \qquad 1<a<3/2,
\]
for all sufficiently large $\chi(G)$.  In each fixed polynomial-density regime we also obtain
\[
        f_{P,C}(k,r)\le k^{O_{r,P,C}(1)}.
\]

The proof combines a chromatic-defect random extraction lemma, compact and near-quadratic sparse-core bases, and a peeling/thinning bootstrap increasing the admissible edge exponent by $1/(r-1)$.  We also prove structural saturation results for possible counterexamples, including Moore-strength exact-cycle packings and quadratic saturation in projected colour-pair space.

Finally, writing
\[
        h_r^{\mathrm f}(G)=\max\{\chi_{\mathrm f}(H):H\subseteq G,\ \girth(H)\ge r\},
\]
we develop a fractional random-extraction framework based on Mohar--Wu preservation.  We prove sufficient cheap-cycle-killing criteria and verify them for several structured families, including clique-organised families, line graphs of incidence graphs of equal-order generalized quadrangles and generalized hexagons, and the Bohman--Keevash tracking-time triangle-free-process graph.  We also isolate a density-free obstruction that any proof using this fractional surgery route must overcome.
\end{abstract}

\maketitle

\section{Introduction}

Erd\H{o}s and Hajnal asked whether, for every pair $r\ge4$ and $k\ge2$, there is a finite number $f(k,r)$ such that every graph $G$ with $\chi(G)\ge f(k,r)$ contains a subgraph $H$ satisfying
\[
        \girth(H)\ge r,\qquad \chi(H)\ge k.
\]
The case $r=4$ is the theorem of R\"odl that sufficiently high chromatic number forces a triangle-free subgraph of large chromatic number \cite{Rodl1977}.  The general problem remains open.  Pettie, Tardos and Walczak recently proved tower-type lower bounds for the first open girth case $r=5$ using Burling graphs \cite{PTW2026}.  Random subgraphs of high-chromatic graphs have also been studied explicitly as a natural route toward the Erd\H{o}s--Hajnal problem; see Bukh, Krivelevich and Narayanan \cite{BKN2023} and the earlier work of Shinkar \cite{Shinkar2016}.  Mohar and Wu proved the fractional analogue of R\"odl's triangle-free subgraph theorem and also proved a fractional random-subgraph preservation theorem \cite{MoharWuTriangleFreeFractional,MoharWuRandomFractional}.

For a graph $G$, write
\[
        h_r(G)=\max\{\chi(H): H\subseteq G,\ \girth(H)\ge r\}.
\]
The Erd\H{o}s--Hajnal problem asks whether $h_r(G)\to\infty$ as $\chi(G)\to\infty$, for each fixed $r$.

\subsection{Main theorem}

The main result proves the Erd\H{o}s--Hajnal conclusion in every fixed polynomial edge-density regime.

\begin{theorem}[Polynomially sparse graphs]\label{thm:poly-sparse}
Fix integers $r\ge4$ and $k\ge2$, and real numbers $P>0$ and $C>0$.  There exists $M=M(r,k,P,C)$ such that every graph $G$ with
\[
        \chi(G)\ge M,\qquad e(G)\le C\chi(G)^P
\]
contains a subgraph $H$ with
\[
        \girth(H)\ge r,
        \qquad
        \chi(H)\ge k.
\]
\end{theorem}

This does not solve the full Erd\H{o}s--Hajnal problem, because a chromatic-critical subgraph of a high-chromatic graph may have super-polynomially many edges in its own chromatic number.  It shows instead that any counterexample family must fail to contain arbitrarily high-chromatic critical cores whose edge count is bounded by a fixed power of their chromatic number; it is not merely a statement about the ambient graph's original edge count.

The frontier case $r=5$, $k=4$ is worth stating explicitly.

\begin{corollary}[The $r=5$, $k=4$ sparse frontier]\label{cor:r5k4}
For every $P>0$ and $C>0$ there is $M=M(P,C)$ such that every graph $G$ with
\[
        \chi(G)\ge M,
        \qquad
        e(G)
        \le C\chi(G)^P
\]
contains a subgraph of girth at least $5$ and chromatic number at least $4$.  In particular, every sufficiently high-chromatic triangle-free graph satisfying the same polynomial edge bound contains a $C_4$-free $4$-chromatic subgraph.
\end{corollary}

The proof of Theorem \ref{thm:poly-sparse} is in Section \ref{sec:poly}.  Its base case is a near-quadratic sparse-core theorem derived from random thinning.  Its induction step is a high-degree peeling dichotomy.  The induction increment is exactly $1/(r-1)$; iterating from the near-quadratic base reaches every fixed exponent $P$.

\subsection{Cycle-profile and core extraction}

The first production lemma gives a direct lower bound for $h_r(G)$ from global short-cycle counts.  Let $C_\ell(G)$ be the number of cycles of length $\ell$ in $G$.

\begin{theorem}[Cycle-profile extraction]\label{thm:cycle-profile}
For every integer $r\ge4$ there is a constant $c_r>0$ such that every graph $G$ with $m=\chi(G)$ satisfies
\[
 h_r(G)\ge c_r\min\left\{m,
 \min_{3\le \ell<r}\left(\frac{m^\ell}{C_\ell(G)+1}\right)^{1/(\ell-1)}\right\}.
\]
\end{theorem}

Consequently, polynomial savings in short-cycle counts produce polynomial high-girth chromatic growth.

\begin{corollary}[Cycle-count savings]\label{cor:cycle-savings}
Fix $r\ge4$.  Suppose $G$ has chromatic number $m$ and, for each $3\le\ell<r$,
\[
        C_\ell(G)\le A_\ell m^{\ell-\eps_\ell}
\]
with $A_\ell\ge1$ and $\eps_\ell>0$.  Then
\[
        h_r(G)\ge c_{r,A}\,m^{\theta},
        \qquad
        \theta=\min\left\{1,\min_{3\le\ell<r}\frac{\eps_\ell}{\ell-1}\right\},
\]
where $c_{r,A}>0$ depends only on $r$ and the numbers $A_\ell$.
\end{corollary}

The same random-thinning method yields compact-core and near-quadratic sparse-core theorems.  These are not needed after Theorem \ref{thm:poly-sparse} for polynomially sparse graphs, but they are the quantitative base of the bootstrapping proof and are useful independently.

\begin{theorem}[Compact chromatic cores]\label{thm:compact-core}
Fix $r\ge4$, $k\ge2$, $B>0$, and
\[
        \beta<\frac{2r-2}{2r-3}.
\]
There is $M=M(r,k,B,\beta)$ such that the following holds.  If a graph $G$ contains a subgraph $J$ with
\[
        \chi(J)=m\ge M,
        \qquad
        |V(J)|\le Bm^\beta,
\]
then $G$ contains a subgraph $H$ with $\girth(H)\ge r$ and $\chi(H)\ge k$.
\end{theorem}

\begin{theorem}[Near-quadratic sparse chromatic cores]\label{thm:edge-core}
Fix $r\ge4$, $k\ge2$, $B>0$, and
\[
        \eps<\frac{2}{3r-5}.
\]
There is $M=M(r,k,B,\eps)$ such that the following holds.  If a graph $G$ contains a subgraph $J$ with
\[
        \chi(J)=m\ge M,
        \qquad
        e(J)\le Bm^{2+\eps},
\]
then $G$ contains a subgraph $H$ with $\girth(H)\ge r$ and $\chi(H)\ge k$.
\end{theorem}

\subsection{Fractional random surgery}

For the fractional chromatic number we write
\[
        h^{\mathrm f}_r(G)=\max\{\chif(H):H\subseteq G,\ \girth(H)\ge r\}.
\]
The fractional analogue has a different bottleneck from the integer argument.  Mohar--Wu's random-subgraph theorem gives small-retention preservation of order $\rho\chif(G)/\log\chif(G)$, with no dependence on $|V(G)|$.  In Section \ref{sec:limits} we combine this theorem with a fractional product-deletion inequality to prove a random-surgery criterion, Theorem \ref{thm:fractional-surgery}.  The criterion shows that, within this random-surgery framework, the remaining obstacle to the full higher-girth fractional Erd\H{o}s--Hajnal conjecture is density-free control of the fractional chromatic cost of the edges used to hit all short cycles.  We also prove a cheap-cycle-killing spread criterion and verify it for random thinnings of complete graphs, which gives a proof of concept for this route to the full fractional conjecture.

\subsection{Structural saturation}

The production theorems are complemented by structural statements for hypothetical counterexamples.  The first is an exact-cycle packing theorem in the first-failing-girth setup.  Let $s\ge3$ and $q\ge1$.  A graph $G$ is called an $(s,q)$-first-failing graph if
\begin{enumerate}[label=\textup{(\roman*)}]
\item $\girth(G)\ge s$, and
\item every subgraph of $G$ containing no cycle of length exactly $s$ is $q$-colourable.
\end{enumerate}
Let $M_s(d)$ be the Moore lower bound for graphs of girth at least $s$ and minimum degree at least $d$:
\[
M_s(d)=
\begin{cases}
1+d\sum_{i=0}^{a-1}(d-1)^i, & s=2a+1,\\[2mm]
2\sum_{i=0}^{a-1}(d-1)^i, & s=2a.
\end{cases}
\]

\begin{theorem}[Moore-strength exact-cycle packing]\label{thm:moore-packing}
Let $G$ be an $(s,q)$-first-failing graph, and let $Y\subseteq G$ have $\chi(Y)=h$.
Then $Y$ contains a family of vertex-disjoint cycles of length $s$ of size at least
\[
        \frac1s M_s(h-q-1),
\]
with the bound interpreted as $0$ if $h-q-1<1$.  Moreover, if $c=\lceil h/q\rceil$, then $Y$ contains a family of edge-disjoint cycles of length $s$ of size at least
\[
        \frac{c-1}{2s}M_s(c-1),
\]
again interpreted as $0$ when $c-1<1$.
\end{theorem}

The second structural theorem works in every counterexample.  Let $K$ be an $m$-chromatic graph, fix a proper surjective $m$-colouring $\varphi:V(K)\to[m]$, and for a short cycle $C$ define
\[
        \Pi(C)=\bigl\{\{\varphi(x),\varphi(y)\}:xy\in E(C)\bigr\}
        \subseteq \binom{[m]}2.
\]

\begin{theorem}[Projected colour-pair saturation]\label{thm:projected}
Let $r\ge4$ and $k\ge2$, and let $K$ be an $m$-chromatic graph with $h_r(K)<k$.  Fix a proper surjective $m$-colouring $\varphi:V(K)\to[m]$, and put $a=k-1$.
\begin{enumerate}[label=\textup{(\alph*)}]
\item If $\eta:V(K)\to[a]$ is any map and $P_\eta$ is the graph on $[m]$ in which $ij$ is an edge when some edge of $K$ between colour classes $i$ and $j$ is $\eta$-monochromatic, then
\[
        \alpha(P_\eta)\le a.
\]
Consequently,
\[
        e(P_\eta)\ge \frac{m^2}{2a}-\frac m2.
\]
\item If $T$ is a graph on $[m]$ meeting $\Pi(C)$ for every cycle $C$ of $K$ of length less than $r$, then
\[
        \alpha(T)\le a
\]
and hence
\[
        e(T)\ge \frac{m^2}{2a}-\frac m2.
\]
\item The graph $K$ contains a family $\calM$ of cycles of length less than $r$ such that the sets $\Pi(C)$, $C\in\calM$, are pairwise disjoint and
\[
        |\calM|\ge \frac1{r-1}\left(\frac{m^2}{2a}-\frac m2\right).
\]
These cycles are pairwise edge-disjoint in $K$.
\end{enumerate}
\end{theorem}

For comparison with the Burling-graph lower-bound examples of Pettie, Tardos and Walczak, the theorem gives the following concrete saturation statement.  If a triangle-free $m$-chromatic graph satisfies $h_5(G)<k$, then, with respect to every proper surjective $m$-colouring, it has at least
\[
        \frac14\left(\frac{m^2}{2(k-1)}-\frac m2\right)
\]
edge-disjoint $4$-cycles whose projected colour-pair sets are pairwise disjoint.  If one only knows $h_5(G)\le k$, as in the usual formulation of the PTW lower-bound examples, the same statement applied with target $k+1$ gives the denominator $2k$ instead.

\subsection{Conventions and notation}

We close the introduction by fixing the conventions used throughout the paper.  All graphs are finite and simple.  A subgraph is not required to be induced unless this is explicitly stated.  Cycles are simple and counted as unoriented subgraphs; $C_\ell(G)$ denotes the number of $\ell$-cycles of $G$.  The girth of a forest is taken to be $\infty$.  All logarithms are natural logarithms, except where a base is displayed.  We use the conventions $\chi(\emptyset)=0$ and $\chi_{\mathrm f}(\emptyset)=0$; a non-empty edgeless graph has ordinary and fractional chromatic number $1$.  For $p\in(0,1)$, $G_p$ denotes the random spanning subgraph obtained by retaining each edge independently with probability $p$.  Whenever an optimal colouring of an $m$-chromatic graph is fixed, it is taken to be a proper surjective $m$-colouring, so no colour class is empty.

\section{Chromatic defect and deterministic deletion}\label{sec:defect}

For an integer $q\ge1$ and a graph $G$, define the $q$-colour defect
\[
        \mu_q(G)=\min_{\psi:V(G)\to[q]}|
        \{uv\in E(G):\psi(u)=\psi(v)\}|.
\]
Thus $G$ is $q$-colourable if and only if $\mu_q(G)=0$.

For integers $m\ge1$ and $q\ge1$, write $m=aq+b$ with $a\ge0$ and $0\le b<q$, and define
\[
        D_q(m)=(q-b)\binom a2+b\binom{a+1}2.
\]
Equivalently, $D_q(m)$ is the minimum of $\sum_{i=1}^q\binom{x_i}{2}$ over nonnegative integers $x_i$ with $\sum_i x_i=m$.

\begin{lemma}[Chromatic defect]\label{lem:defect}
For every graph $G$ with $\chi(G)=m$ and every integer $q\ge1$,
\[
        \mu_q(G)\ge D_q(m).
\]
In particular, if $m\ge q$, then
\[
        \mu_q(G)\ge \frac{(m-q)^2}{2q}.
\]
\end{lemma}

\begin{proof}
Fix a map $\psi:V(G)\to[q]$ and let $V_i=\psi^{-1}(i)$.  Put $t_i=\chi(G[V_i])$.  Colouring each induced graph $G[V_i]$ with a disjoint palette gives a colouring of $G$ with $\sum_i t_i$ colours, so
\[
        \sum_{i=1}^q t_i\ge m.
\]
Every graph of chromatic number $t$ has at least $\binom t2$ edges: take a $t$-critical subgraph, whose minimum degree is at least $t-1$.  Hence $G[V_i]$ has at least $\binom{t_i}{2}$ edges, all $\psi$-monochromatic.  Therefore the number of $\psi$-monochromatic edges is at least
\[
        \sum_{i=1}^q\binom{t_i}{2}.
\]
This is minimized, under $\sum_i t_i\ge m$, when the $t_i$ are as equal as possible and have sum exactly $m$, giving $D_q(m)$.

For the quadratic bound, Cauchy's inequality gives
\[
\sum_i\binom{t_i}{2}
=\frac12\left(\sum_i t_i^2-\sum_i t_i\right)
\ge \frac12\left(\frac{m^2}{q}-m\right)
=\frac{m(m-q)}{2q}
\ge \frac{(m-q)^2}{2q}
\]
when $m\ge q$.
\end{proof}

A set $F\subseteq E(G)$ is an $r$-short-cycle transversal if $G-F$ has girth at least $r$.  Let $\tau_{<r}(G)$ be the minimum size of such a set, and put
\[
        C_{<r}(G)=\sum_{\ell=3}^{r-1}C_\ell(G).
\]

\begin{lemma}[Deletion criterion]\label{lem:deletion}
If $\tau_{<r}(G)<\mu_q(G)$, then $h_r(G)\ge q+1$.  In particular, if $C_{<r}(G)<\mu_q(G)$, then $h_r(G)\ge q+1$.
\end{lemma}

\begin{proof}
Let $F$ be an $r$-short-cycle transversal with $|F|=\tau_{<r}(G)$, and set $G'=G-F$.  If $\chi(G')\le q$, then a proper $q$-colouring of $G'$ is a $q$-colouring of $G$ in which every monochromatic edge lies in $F$.  Hence $\mu_q(G)\le |F|$, a contradiction.  Therefore $\chi(G')\ge q+1$, and $G'$ has girth at least $r$.

For the final assertion, delete one edge from each cycle of length less than $r$.
\end{proof}

\begin{corollary}[A deterministic lower bound for $h_r$]\label{cor:deterministic-hr}
There is an absolute constant $c_0>0$ such that every graph $G$ with $m=\chi(G)$ and $T=C_{<r}(G)$ satisfies
\[
        h_r(G)\ge c_0\min\left\{m,\frac{m^2}{T+1}\right\}.
\]
\end{corollary}

\begin{proof}
If the displayed minimum is bounded by an absolute constant, the assertion is trivial after decreasing $c_0$.  Otherwise choose
\[
        q=\left\lfloor \min\left\{\frac m2,\frac{m^2}{16(T+1)}\right\}\right\rfloor\ge1.
\]
Since $q\le m/2$, Lemma \ref{lem:defect} gives
\[
        \mu_q(G)\ge \frac{(m-q)^2}{2q}\ge \frac{m^2}{8q}.
\]
By the choice of $q$, $m^2/(8q)>T$ after a harmless adjustment of constants.  Lemma \ref{lem:deletion} gives $h_r(G)\ge q+1$, which is a fixed positive multiple of $\min\{m,m^2/(T+1)\}$.
\end{proof}

\section{Random edge extraction}\label{sec:random}

The next lemma samples edges independently and then deletes one edge from each surviving short cycle.  The formulation is designed so that the chromatic number is preserved through chromatic defect.

\begin{theorem}[Robust random extraction]\label{thm:random-extraction}
Let $r\ge4$ and $q\ge1$.  Let $J$ be a graph with
\[
        n=|V(J)|,
        \qquad
        m=\chi(J)>q,
\]
and put
\[
        M_q=\frac{(m-q)^2}{2q}.
\]
Suppose there is $p\in(0,1]$ such that
\begin{align}
        pM_q&\ge 16(n\log q+\log 4),\label{eq:random-preserve}\\
        \sum_{\ell=3}^{r-1}p^\ell C_\ell(J)&\le \frac{pM_q}{16}.\label{eq:random-cycles}
\end{align}
Then $J$ contains a subgraph $H$ with $\girth(H)\ge r$ and $\chi(H)\ge q+1$.
\end{theorem}

\begin{proof}
Let $J_p$ be the random spanning subgraph obtained from $J$ by keeping each edge independently with probability $p$.

Fix a map $\psi:V(J)\to[q]$.  By Lemma \ref{lem:defect}, the graph $J$ has an integer number $N_\psi\ge M_q$ of $\psi$-monochromatic edges.  The number of such edges surviving in $J_p$ is binomial with mean $\mu_\psi=pN_\psi\ge pM_q$.  Since $pM_q/2\le \mu_\psi/2$, Chernoff's inequality gives
\[
\Prob\bigl(\text{fewer than }pM_q/2\text{ such edges survive}\bigr)
        \le \Prob(\operatorname{Bin}(N_\psi,p)<\mu_\psi/2)
        \le \exp(-pM_q/8).
\]
There are $q^n$ maps $V(J)\to[q]$.  By \eqref{eq:random-preserve}, with probability greater than $3/4$, every $q$-colouring of $V(J)$ has at least $pM_q/2$ monochromatic edges in $J_p$.

The expected number of cycles of length less than $r$ in $J_p$ is
\[
        \sum_{\ell=3}^{r-1}p^\ell C_\ell(J)\le pM_q/16.
\]
By Markov's inequality, with probability at least $3/4$, the graph $J_p$ has at most $pM_q/4$ cycles of length less than $r$.  Hence some outcome satisfies both properties.  Delete one edge from each short cycle in this outcome, producing $H$.  Then $\girth(H)\ge r$.  Since at most $pM_q/4$ edges were deleted, every $q$-colouring still has a monochromatic edge in $H$.  Thus $\chi(H)\ge q+1$.
\end{proof}

\section{The cycle-profile bound}\label{sec:profile}

We now prove Theorem \ref{thm:cycle-profile}.  The key point is that if the vertex set is randomly partitioned into $t$ parts, a fixed $\ell$-cycle survives inside one part with probability $t^{1-\ell}$, while for every partition the sum of the chromatic numbers of the induced parts is at least $\chi(G)$.

\begin{proof}[Proof of Theorem \ref{thm:cycle-profile}]
Let $m=\chi(G)$ and define
\[
        Q=
        \min\left\{m,
        \min_{3\le\ell<r}\left(\frac{m^\ell}{C_\ell(G)+1}\right)^{1/(\ell-1)}
        \right\}.
\]
It is enough to show $h_r(G)\ge c_rQ$.  If $Q<1$, this is trivial.  Assume $Q\ge1$ and set
\[
        t=\left\lceil\frac mQ\right\rceil.
\]
Randomly partition $V(G)$ into $t$ labelled parts $V_1,\dots,V_t$, each vertex choosing its part independently and uniformly.  Let $G_i=G[V_i]$, $s_i=\chi(G_i)$, and $T_i=C_{<r}(G_i)$.

For a fixed $\ell$-cycle, the probability that all its vertices lie in a common part is $t^{1-\ell}$.  Hence
\[
        \E\left[\sum_iT_i\right]
        =\sum_{\ell=3}^{r-1}C_\ell(G)t^{1-\ell}.
\]
Since $t\ge m/Q$ and $C_\ell(G)+1\le m^\ell/Q^{\ell-1}$,
\[
        C_\ell(G)t^{1-\ell}
        \le C_\ell(G)\left(\frac Qm\right)^{\ell-1}
        \le m.
\]
Thus
\[
        \E\left[\sum_iT_i\right]\le (r-3)m.
\]
Choose a partition with $\sum_iT_i\le 2(r-3)m$.

For this fixed partition, colouring each $G_i$ with a disjoint palette gives a colouring of $G$, so
\[
        \sum_i s_i\ge m.
\]
Let $I=\{i:s_i\ge m/(2t)\}$.  The parts outside $I$ contribute less than $m/2$, so $\sum_{i\in I}s_i\ge m/2$.  Moreover,
\[
        \sum_{i\in I}(T_i+1)\le 2(r-3)m+t\le 2(r-2)m,
\]
since $t\le 2m/Q\le2m$.  Therefore some $i\in I$ satisfies
\[
        \frac{T_i+1}{s_i}\le 4(r-2).
\]
Corollary \ref{cor:deterministic-hr} gives
\[
        h_r(G)\ge h_r(G_i)
        \ge c_0\min\left\{s_i,\frac{s_i^2}{T_i+1}\right\}
        \ge \frac{c_0}{4(r-2)}s_i.
\]
Finally, $s_i\ge m/(2t)\ge Q/4$.  This proves the theorem with $c_r=c_0/(16(r-2))$.
\end{proof}

\begin{proof}[Proof of Corollary \ref{cor:cycle-savings}]
For each $3\le \ell<r$ put
\[
        \theta_\ell=\min\left\{1,\frac{\eps_\ell}{\ell-1}\right\}.
\]
We first prove the correctly capped estimate for each cycle length.  Since $m\ge1$ and $A_\ell\ge1$,
\[
        C_\ell(G)+1
        \le A_\ell m^{\ell-\eps_\ell}+1
        \le (A_\ell+1)m^{\ell-(\ell-1)\theta_\ell}.        \tag{1}
\]
Indeed, if $\eps_\ell\le \ell-1$, then $(\ell-1)\theta_\ell=\eps_\ell$ and (1) follows from $1\le m^{\ell-\eps_\ell}$.  If $\eps_\ell>\ell-1$, then $\theta_\ell=1$ and $m^{\ell-\eps_\ell}\le m$, so (1) again follows.
Consequently
\[
        \left(\frac{m^\ell}{C_\ell(G)+1}\right)^{1/(\ell-1)}
        \ge (A_\ell+1)^{-1/(\ell-1)}m^{\theta_\ell}.
\]
Theorem \ref{thm:cycle-profile} also contains the separate cap by $m$, which is exactly the exponent-$1$ cap.  Taking the minimum over $\ell$ and decreasing the constant to absorb the finitely many factors depending on $A_\ell$ gives
\[
        h_r(G)
        \ge c_{r,A}m^{\min\{1,\min_{3\le\ell<r}\eps_\ell/(\ell-1)\}}.
\]
This proves the claimed bound.
\end{proof}

\section{Compact and sparse chromatic cores}\label{sec:cores}

We shall use the following standard spectral estimate.

\begin{lemma}[Spectral cycle bound]\label{lem:spectral}
If $G$ has $e$ edges, then for every integer $\ell\ge3$,
\[
        C_\ell(G)\le (2e)^{\ell/2}.
\]
\end{lemma}

\begin{proof}
Let $A$ be the adjacency matrix of $G$, with eigenvalues $\lambda_1,\dots,\lambda_n$.  The number of closed walks of length $\ell$ is at most $\sum_i|\lambda_i|^\ell$.  Since $\ell\ge2$,
\[
        \sum_i|\lambda_i|^\ell
        \le \left(\sum_i\lambda_i^2\right)^{\ell/2}
        =(2e)^{\ell/2}.
\]
Every $\ell$-cycle contributes at least one closed walk of length $\ell$, so the claimed bound follows.
\end{proof}

\begin{proof}[Proof of Theorem \ref{thm:compact-core}]
It suffices to work inside $J$.  Put $q=k-1$.  The case $k=2$ is trivial for large $m$, so assume $q\ge2$.  Let $n=|V(J)|\le Bm^\beta$.  Since $\beta<(2r-2)/(2r-3)$, we can choose $\alpha>0$ such that
\[
        \max_{3\le\ell<r}\frac{\beta\ell-2}{\ell-1}<\alpha<2-\beta.
\]
The maximum on the left is at $\ell=r-1$, and the displayed interval is nonempty exactly under the stated hypothesis on $\beta$.

Set $p=m^{-\alpha}$ and $M_q=(m-q)^2/(2q)$.  Since $\alpha<2-\beta$, we have $pM_q/n\to\infty$, so \eqref{eq:random-preserve} holds for all sufficiently large $m$.  For $3\le\ell<r$, the trivial bound $C_\ell(J)\le n^\ell$ gives
\[
        p^\ell C_\ell(J)\le B^\ell m^{\beta\ell-\alpha\ell}.
\]
The choice of $\alpha$ gives $\beta\ell-\alpha\ell<2-\alpha$, so $p^\ell C_\ell(J)=o(pM_q)$.  Hence \eqref{eq:random-cycles} holds for large $m$, and Theorem \ref{thm:random-extraction} applies.
\end{proof}

\begin{proof}[Proof of Theorem \ref{thm:edge-core}]
Again it suffices to work inside $J$.  Put $q=k-1$ and assume $q\ge2$, the case $k=2$ being trivial.  Replace $J$ by an $m$-critical subgraph.  Then $e(J)$ does not increase and $\chi(J)=m$.  Since an $m$-critical graph has minimum degree at least $m-1$,
\[
        |V(J)|\le \frac{2e(J)}{m-1}\le B'm^{1+\eps}
\]
for large $m$.

Since $\eps<2/(3r-5)$, choose $\alpha$ satisfying
\[
        \max_{3\le\ell<r}\frac{(1+\eps/2)\ell-2}{\ell-1}<\alpha<1-\eps.
\]
The left side is largest at $\ell=r-1$, and the interval is nonempty exactly when $\eps<2/(3r-5)$.

Set $p=m^{-\alpha}$ and $M_q=(m-q)^2/(2q)$.  Because $\alpha<1-\eps$, condition \eqref{eq:random-preserve} holds for all sufficiently large $m$.  By Lemma \ref{lem:spectral} and $e(J)\le Bm^{2+\eps}$,
\[
        C_\ell(J)\le (2B)^{\ell/2}m^{(1+\eps/2)\ell}.
\]
The choice of $\alpha$ gives $p^\ell C_\ell(J)=o(pM_q)$ for every $\ell<r$, so \eqref{eq:random-cycles} also holds for large $m$.  Theorem \ref{thm:random-extraction} completes the proof.
\end{proof}

\section{Bootstrapping to all polynomial edge exponents}\label{sec:poly}

This section proves Theorem \ref{thm:poly-sparse}.  The proof uses Theorem \ref{thm:edge-core} as a base case and then bootstraps the allowable edge exponent by increments of $1/(r-1)$.

Fix $r\ge4$ and $k\ge2$.  For a real number $A>0$, let $\mathbf S(A)$ denote the following statement:

\begin{quote}
For every $x<A$ and every $C>0$, there is $M=M(r,k,x,C)$ such that every graph $G$ with $\chi(G)\ge M$ and $e(G)\le C\chi(G)^x$ satisfies $h_r(G)\ge k$.
\end{quote}

\begin{lemma}[Base exponent]\label{lem:base-S}
The statement $\mathbf S(A_0)$ holds for
\[
        A_0=2+\frac{2}{3r-5}.
\]
\end{lemma}

\begin{proof}
Let $x<A_0$.  If $x<2$, then no graph of sufficiently large chromatic number can satisfy $e(G)\le C\chi(G)^x$, because an $m$-critical graph has at least $m(m-1)/2$ edges.  If $x\ge2$, write $x=2+\eps$ with $\eps<2/(3r-5)$ and apply Theorem \ref{thm:edge-core}.
\end{proof}

The induction step is based on the next mixed vertex-edge proposition.

\begin{proposition}[Mixed peeling and thinning]\label{prop:mixed}
Assume $\mathbf S(A)$ for some $A>0$.  Let $x>A$, $y\ge0$, and $C>0$ satisfy
\[
        y<(r-1)A-(r-2)x.        \tag{\(*\)}\label{eq:mixed-condition}
\]
Then there is $M=M(r,k,A,x,y,C)$ such that every graph $G$ with
\[
        \chi(G)=t\ge M,
        \qquad
        e(G)\le Ct^x,
        \qquad
        |V(G)|\le Ct^y
\]
satisfies $h_r(G)\ge k$.
\end{proposition}

\begin{proof}
The case $k=2$ is trivial for large $t$, so assume $q:=k-1\ge2$.  Since $x>A$, condition \eqref{eq:mixed-condition} implies $y<A$.  The slack in the cycle-count inequalities is completely explicit.  For $3\le \ell<r$, put
\[
        \sigma_\ell=\ell A-(\ell-1)x-y.
\]
Then
\[
        \sigma_{\ell+1}-\sigma_\ell=A-x<0,
\]
so the smallest slack occurs at the endpoint $\ell=r-1$.  Condition \eqref{eq:mixed-condition} says exactly that $\sigma_{r-1}>0$, and hence $\sigma_\ell>0$ for every $3\le \ell<r$, including the other endpoint $\ell=3$.

Choose $\eta>0$ such that
\begin{align}
        \eta&<A-y,\label{eq:eta1}\\
        (2\ell-2)\eta&<\sigma_\ell
        \qquad(3\le\ell<r).\label{eq:eta2}
\end{align}
Equivalently,
\[
        y+(\ell-1)x+(2\ell-2)\eta<\ell A
        \qquad(3\le\ell<r).
\]
Put
\[
        \gamma=A-\eta,
        \qquad
        d(u)=x-u+\eta.
\]
For every $u\le y$ we have $u<\gamma$ and $d(u)>0$.

We now state the finite inner induction with all parameters fixed.  In order to keep the later quantitative thresholds under control, the inner induction is terminated at a fixed compact-core exponent rather than by converting a small vertex bound into an edge bound.  Fix once and for all
\[
        \beta_*:=1+\frac1{4r}<\frac{2r-2}{2r-3}.
\]
For $u\in[0,y]$ and $C'>0$, let $\mathcal P(u,C')$ be the assertion that all sufficiently large graphs $G$ with
\[
        \chi(G)=t,\qquad e(G)\le C't^x,\qquad |V(G)|\le C't^u
\]
satisfy $h_r(G)\ge k$.  The threshold allowed in $\mathcal P(u,C')$ may depend on
\[
        r,k,A,x,y,\eta,u,C',
\]
but not on $t$.  We prove $\mathcal P(u,C')$ for every $u\in[0,y]$ by induction on
\[
        N(u)=\min\{N\ge0:u-N\eta<\beta_*\}.
\]
This is a finite induction for each fixed $u$.

If $N(u)=0$, then $u<\beta_*$.  Since $t\ge1$, we have $|V(G)|\le C't^{\beta_*}$.  The compact-core theorem, Theorem \ref{thm:compact-core}, applied with $B=C'$ and $\beta=\beta_*$, gives $\mathcal P(u,C')$.  Notice that this base case uses the vertex bound directly.  It does not replace $e(G)$ by $|V(G)|^2/2$, and therefore it does not square the constant $C'$.

Assume $N(u)>0$ and that $\mathcal P(u',C'')$ is already known whenever $N(u')<N(u)$ and $C''>0$ is arbitrary.  Let $G$ satisfy the hypotheses of $\mathcal P(u,C')$, with $t$ larger than all thresholds specified below.  Set
\[
        d=d(u)=x-u+\eta,
        \qquad
        Z=\{v\in V(G):d_G(v)>t^d\},
\]
and put $R=G-Z$.  Since $e(G)\le C't^x$,
\[
        |Z|\le 2C't^{x-d}=2C't^{u-\eta}.
\]
If $\chi(G[Z])\ge t/2$, write $t_Z=\chi(G[Z])$.  Then $t_Z\ge t/2$,
\[
        e(G[Z])\le C'2^x t_Z^x,
        \qquad
        |V(G[Z])|\le 2C'2^{u-\eta}t_Z^{u-\eta}.
\]
If $u-\eta<0$, then $|Z|<1$ for all sufficiently large $t$, so this high-degree branch is impossible.  Otherwise $u-\eta\in[0,y]$ and $N(u-\eta)<N(u)$, so the induction hypothesis $\mathcal P(u-\eta,C_Z)$, with
\[
        C_Z=\max\{C'2^x,2C'2^{u-\eta}\},
\]
gives $h_r(G[Z])\ge k$ once $t$ is sufficiently large.  Thus we may assume
\[
        \chi(G[Z])<t/2.
\]
Then
\[
        \chi(R)>t/2,
        \qquad
        \Delta(R)\le t^d,
        \qquad
        |V(R)|\le C't^u,
        \qquad
        e(R)\le C't^x.
\]

Suppose first that some map $\psi:V(R)\to[q]$ has at most $t^\gamma$ monochromatic edges.  Let $M_\psi$ be the graph of these monochromatic edges, and put $s=\chi(M_\psi)$.  If $M_\psi$ were $s$-colourable with $s<t/(2q)$, then the ordered pair consisting of the $\psi$-colour and the $s$-colour would properly colour $R$ with fewer than $t/2$ colours, contrary to $\chi(R)>t/2$.  Hence
\[
        s\ge t/(2q).
\]
Also
\[
        e(M_\psi)\le t^\gamma\le (2q+2)^\gamma s^\gamma.
\]
Since $\gamma=A-\eta<A$, the statement $\mathbf S(A)$, applied with exponent $\gamma$ and constant $(2q+2)^\gamma$, gives $h_r(M_\psi)\ge k$ for all sufficiently large $t$.

It remains to handle the case in which every map $V(R)\to[q]$ has more than $t^\gamma$ monochromatic edges.  Let
\[
        p=L t^{u-\gamma},
\]
where $L>16C'\log q+16$ is fixed.  Since $u<\gamma$, $p\le1$ for all large $t$.  Let $R_p$ be the random spanning subgraph of $R$ obtained by keeping each edge independently with probability $p$.

For a fixed map $\psi:V(R)\to[q]$, the number of surviving $\psi$-monochromatic edges has mean at least $Lt^u$.  By Chernoff's inequality, the probability that it is less than $Lt^u/2$ is at most $\exp(-Lt^u/8)$.  There are at most $q^{C't^u}$ maps $V(R)\to[q]$.  Our choice of $L$ ensures that, with probability at least $3/4$ for all large $t$, every $q$-colouring has at least $Lt^u/2$ monochromatic edges in $R_p$.

For $3\le\ell<r$, the number of $\ell$-cycles in $R$ is at most $2e(R)\Delta(R)^{\ell-2}$.  Therefore the expected number of cycles of length less than $r$ in $R_p$ is at most a constant depending on $r$ and $C'$ times
\[
\sum_{\ell=3}^{r-1}
        p^\ell t^x t^{d(\ell-2)}.
\]
The exponent of $t$ in the $\ell$th summand is
\begin{align*}
        \ell(u-\gamma)+x+d(\ell-2)
        &=\ell(u-A+\eta)+x+(x-u+\eta)(\ell-2)\\
        &=2u+(\ell-1)x-\ell A+(2\ell-2)\eta.
\end{align*}
Since $u\le y$, \eqref{eq:eta2} gives
\[
        u+(\ell-1)x+(2\ell-2)\eta
        \le y+(\ell-1)x+(2\ell-2)\eta
        <\ell A.
\]
Thus the displayed exponent is strictly smaller than $u$.  Consequently the expected number of short cycles in $R_p$ is $o(t^u)$.  For all large $t$, Markov's inequality gives probability at least $3/4$ that $R_p$ has fewer than $Lt^u/4$ cycles of length less than $r$.

Choose an outcome satisfying both high-probability events.  Delete one edge from each cycle of length less than $r$, obtaining $H$.  Then $\girth(H)\ge r$.  Every $q$-colouring of $V(R)$ had at least $Lt^u/2$ monochromatic edges before the deletion, and fewer than $Lt^u/4$ edges were deleted.  Hence every $q$-colouring still has a monochromatic edge in $H$, so $\chi(H)\ge q+1=k$.  This proves $\mathcal P(u,C')$.

Taking $u=y$ and $C'=C$ gives the proposition.
\end{proof}

\begin{lemma}[Exponent increment]\label{lem:increment}
If $\mathbf S(A)$ holds, then $\mathbf S(A+1/(r-1))$ holds.
\end{lemma}

\begin{proof}
Let $x<A+1/(r-1)$ and $C>0$.  If $x<A$, this is exactly $\mathbf S(A)$.  Assume $x\ge A$.  Choose a number $\bar x$ such that
\[
        \max\{A,x\}<\bar x<A+\frac1{r-1}.
\]
This harmless enlargement is only needed when $x=A$, because Proposition \ref{prop:mixed} assumes a strict inequality $\bar x>A$.  Let $G$ be a graph with $\chi(G)=m$ and $e(G)\le Cm^x$.  Pass to an $m$-critical subgraph $G_0\subseteq G$.  Then $e(G_0)\le Cm^x\le Cm^{\bar x}$ for $m\ge1$, and, since $\delta(G_0)\ge m-1$,
\[
        |V(G_0)|\le \frac{2e(G_0)}{m-1}\le C_1m^{x-1}
        \le C_1m^{\bar x-1}
\]
for a constant $C_1$ depending on $C$ and $x$.  Set $y=\bar x-1$ and $C_2=\max\{C,C_1\}$.
The inequality $\bar x<A+1/(r-1)$ is equivalent to
\[
        \bar x-1<(r-1)A-(r-2)\bar x.
\]
Thus Proposition \ref{prop:mixed}, applied with edge exponent $\bar x$, vertex exponent $y$, and constant $C_2$, gives $h_r(G_0)\ge k$ for all sufficiently large $m$.  Since $G_0\subseteq G$, this proves $\mathbf S(A+1/(r-1))$.
\end{proof}

\begin{proof}[Proof of Theorem \ref{thm:poly-sparse}]
By Lemma \ref{lem:base-S}, $\mathbf S(A_0)$ holds for $A_0=2+2/(3r-5)$.  By Lemma \ref{lem:increment}, $\mathbf S(A_0+j/(r-1))$ holds for every integer $j\ge0$.  Given $P$, choose $j$ with $P<A_0+j/(r-1)$.  Then $\mathbf S(A_0+j/(r-1))$ applied with $x=P$ gives the desired threshold $M(r,k,P,C)$.
\end{proof}

\begin{remark}[Effectivity]
The thresholds in Theorem \ref{thm:poly-sparse} are effective.  The proof gives them recursively: the base thresholds are obtained from the explicit inequalities in the proof of Theorem \ref{thm:edge-core}; each application of Proposition \ref{prop:mixed} introduces only the constants $\eta$, $L$, Chernoff bounds, Markov bounds and the previously constructed threshold at a smaller vertex exponent.  No compactness or infinitary argument is used.  We have not optimized the resulting bounds, and they should be expected to be very large.
\end{remark}

\begin{remark}[Relation to the Bukh random-subgraph question]
The proof above does not prove the Bukh-type conjectural lower bound $\E\chi(G_p)\gg \chi(G)/\log\chi(G)$ for arbitrary fixed $p$.  That question remains a separate random-subgraph problem; see \cite{BKN2023,BerkowitzDevlinLeeReichardTownley2018}.  Instead, the bootstrapping argument avoids the global $|V|$ loss in the polynomially sparse regime by repeatedly peeling high-degree vertices and applying random thinning only after either a smaller sparse subgraph or a large chromatic defect has been forced.  The fractional random-subgraph analogue of Bukh's question is known by work of Mohar and Wu \cite{MoharWuRandomFractional}.
\end{remark}

\section{The cases \texorpdfstring{$k=2$}{k=2} and \texorpdfstring{$k=3$}{k=3}}\label{sec:smallk}

The cases $k=2$ and $k=3$ are classical.  The $k=3$ statement is an immediate corollary of the Erd\H{o}s--Hajnal odd-cycle theorem: if the longest odd cycle of a non-bipartite graph has length $\ell$, then $\chi(G)\le \ell+1$ \cite{ErdosHajnal1966}.  The extremal equality case was later sharpened by Kenkre and Vishwanathan \cite{KenkreVishwanathan2007}.  We include the short deduction below for completeness.  Depth-colouring proofs of the underlying odd-cycle bounds appear in work of Diwan, Kenkre and Vishwanathan \cite{DiwanKenkreVishwanathan2008}.

\begin{theorem}[Classical exact small targets]\label{thm:small-k}
For every $r\ge4$,
\[
        f(2,r)=2,
\]
and
\[
        f(3,r)=2\left\lceil\frac{r-1}{2}\right\rceil+1
        =
        \begin{cases}
        r,& r\text{ odd},\\
        r+1,& r\text{ even}.
        \end{cases}
\]
\end{theorem}

\begin{proof}
The case $k=2$ is immediate: every graph of chromatic number at least $2$ contains an edge, and an edge has girth $\infty$.

For $k=3$, a subgraph of girth at least $r$ and chromatic number at least $3$ is exactly an odd cycle of length at least $r$.  If $G$ contains no such odd cycle, then its longest odd cycle has length at most $r-2$ when $r$ is odd and at most $r-1$ when $r$ is even.  Erd\H{o}s--Hajnal's odd-cycle theorem gives
\[
        \chi(G)\le
        \begin{cases}
        r-1,& r\text{ odd},\\
        r,& r\text{ even}.
        \end{cases}
\]
Thus the displayed value of $f(3,r)$ is an upper bound.

For sharpness, set
\[
        L=2\left\lceil\frac{r-1}{2}\right\rceil.
\]
Then $L=r-1$ if $r$ is odd and $L=r$ if $r$ is even; in both cases $L$ is even.  The complete graph $K_L$ has chromatic number $L$, and its longest odd cycle has length $L-1$, which is $r-2$ for odd $r$ and $r-1$ for even $r$.  Thus $K_L$ has no odd cycle of length at least $r$, and so it has no subgraph of girth at least $r$ and chromatic number at least $3$.  Hence $f(3,r)>L$, proving equality.
\end{proof}

\section{Structural saturation theorems}\label{sec:structural}

We prove Theorems \ref{thm:moore-packing} and \ref{thm:projected}.

\begin{proof}[Proof of Theorem \ref{thm:moore-packing}]
Let $\mathcal P$ be a maximal family of vertex-disjoint $s$-cycles in $Y$, and let $U$ be the union of their vertex sets.  Then $Y-U$ has no $s$-cycle, so it is $q$-colourable by the first-failing hypothesis.  Hence
\[
        \chi(Y[U])\ge h-q.
\]
The graph $Y[U]$ contains a critical subgraph of chromatic number at least $h-q$, whose minimum degree is at least $h-q-1$.  Since $G$ has girth at least $s$, the Moore bound gives
\[
        |U|\ge M_s(h-q-1).
\]
Because $|U|=s|\mathcal P|$, the vertex-disjoint bound follows.

For the edge-disjoint bound, let $\mathcal Q$ be a maximal family of edge-disjoint $s$-cycles in $Y$, and let $F$ be the union of their edge sets.  Then $F$ meets every $s$-cycle of $Y$, so $Y-F$ has no $s$-cycle and is $q$-colourable.  If the graph $(V(Y),F)$ has chromatic number $b$, then the product of a $q$-colouring of $Y-F$ and a $b$-colouring of $(V(Y),F)$ is a proper colouring of $Y$.  Therefore $h\le qb$, and $b\ge c:=\lceil h/q\rceil$.

The graph $(V(Y),F)$ contains a critical subgraph of chromatic number at least $c$, hence of minimum degree at least $c-1$ and girth at least $s$.  It has at least $M_s(c-1)$ vertices and therefore at least $(c-1)M_s(c-1)/2$ edges.  Since $F$ is the disjoint union of $s$-edge cycles from $\mathcal Q$, we get
\[
        |\mathcal Q|=|F|/s\ge \frac{c-1}{2s}M_s(c-1).
\]
\end{proof}

We need one elementary fact about optimal colourings.

\begin{lemma}[Exact colour subsets]\label{lem:colour-subsets}
Let $K$ be $m$-chromatic and let $\varphi:V(K)\to[m]$ be a proper surjective $m$-colouring.  For every $I\subseteq[m]$,
\[
        \chi\bigl(K[\varphi^{-1}(I)]\bigr)=|I|.
\]
\end{lemma}

\begin{proof}
The left side is at most $|I|$.  If it were smaller for some $I$, then recolouring $K[\varphi^{-1}(I)]$ with fewer than $|I|$ colours and leaving all other colour classes unchanged would give a proper colouring of $K$ with fewer than $m$ colours, a contradiction.
\end{proof}

\begin{proof}[Proof of Theorem \ref{thm:projected}]
For part (a), let $I$ be an independent set in $P_\eta$.  Then no edge of $K[\varphi^{-1}(I)]$ is $\eta$-monochromatic, so $\eta$ is a proper $a$-colouring of this induced subgraph.  By Lemma \ref{lem:colour-subsets}, the subgraph has chromatic number $|I|$.  Thus $|I|\le a$, proving $\alpha(P_\eta)\le a$.  Turan's theorem in the form $e(F)\ge n^2/(2\alpha(F))-n/2$ gives the displayed edge lower bound.

For part (b), first consider $k=2$.  The hypothesis $h_r(K)<2$ implies that $K$ has no edge, so $m=1$ and the assertion is immediate.  We may therefore assume $k\ge3$.  Suppose that $I$ is an independent set in $T$ with $|I|\ge k$.  By Lemma \ref{lem:colour-subsets}, $K[\varphi^{-1}(I)]$ has chromatic number $|I|\ge k$.  If this induced subgraph had girth at least $r$, it would itself contradict $h_r(K)<k$.  Therefore it has girth less than $r$, i.e. it contains a cycle $C$ of length less than $r$.  Its projected set $\Pi(C)$ is contained in $\binom I2$, and hence is disjoint from $T$, contradicting that $T$ meets every projected short cycle.  Thus $\alpha(T)\le k-1=a$, and the same Turan bound applies.

For part (c), take a maximal family $\calM$ of cycles of length less than $r$ whose projected sets $\Pi(C)$ are pairwise disjoint.  Let
\[
        T=\bigcup_{C\in\calM}\Pi(C).
\]
By maximality, $T$ meets $\Pi(D)$ for every short cycle $D$ of $K$.  Part (b) gives
\[
        |T|\ge \frac{m^2}{2a}-\frac m2.
\]
Each cycle of length less than $r$ contributes at most $r-1$ projected pairs, so
\[
        |\calM|\ge \frac{|T|}{r-1}
        \ge \frac1{r-1}\left(\frac{m^2}{2a}-\frac m2\right).
\]
If two cycles in $\calM$ shared an actual edge $uv$, then both projected sets would contain the pair $\{\varphi(u),\varphi(v)\}$, contrary to pairwise disjointness.  Hence the cycles are edge-disjoint.
\end{proof}

\section{Limits of density reduction and fractional random surgery}\label{sec:limits}

Theorem \ref{thm:poly-sparse} naturally suggests trying to reduce the full Erd\H{o}s--Hajnal problem to polynomially sparse subgraphs.  The most naive form of such a reduction is false.  This is not a paradox: high-girth graphs themselves already satisfy the desired high-girth conclusion, but their high-chromatic subgraphs can be forced by Moore bounds to have many more than any prescribed polynomial number of edges in their chromatic number.

\begin{proposition}[No universal fixed-polynomial density reduction]\label{prop:no-density-bridge}
Fix $P>0$ and let $g:\mathbb N\to\mathbb R$ satisfy $g(n)\to\infty$.  There are graphs $G$ of arbitrarily large chromatic number such that every subgraph $H\subseteq G$ with
\[
        \chi(H)\ge g(\chi(G))
\]
satisfies
\[
        e(H)>\chi(H)^P.
\]
Consequently, there is no fixed exponent $P$ and no function $g\to\infty$ such that every graph $G$ contains a subgraph $H$ with $\chi(H)\ge g(\chi(G))$ and $e(H)\le \chi(H)^P$.
\end{proposition}

\begin{proof}
Choose an integer $a>P$.  By Erd\H{o}s's high-girth high-chromatic theorem \cite{Erdos1959}, there are graphs of arbitrarily large chromatic number and girth at least $2a+1$.  Let $G$ be one such graph, chosen so that $g(\chi(G))$ is larger than a constant depending only on $a$ and $P$.

Let $H\subseteq G$ have chromatic number $m\ge g(\chi(G))$.  Take an $m$-critical subgraph $H_0\subseteq H$.  Then $\delta(H_0)\ge m-1$, and $H_0$ has girth at least $2a+1$.  The Moore bound applied to the ball of radius $a$ gives
\[
        |V(H_0)|\ge 1+(m-1)\sum_{i=0}^{a-1}(m-2)^i.
\]
Thus
\[
        e(H)\ge e(H_0)
        \ge \frac{m-1}{2}|V(H_0)|
        \ge c_a m^{a+1}
\]
for some constant $c_a>0$.  Since $a+1>P$, this is larger than $m^P$ for all sufficiently large $m$.  Our choice of $G$ ensures $m\ge g(\chi(G))$ is sufficiently large, completing the proof.
\end{proof}

The preceding proposition shows that the bridge to the full Erd\H{o}s--Hajnal problem cannot be a girth-blind fixed-polynomial density reduction.  Any successful bridge must either allow the exponent to grow, exploit the target girth parameter, or use structural information beyond edge density.

For comparison with the Burling examples, the standard recursive Burling construction is itself super-polynomially large in its chromatic parameter.  We record the elementary calculation because it is a useful compatibility check with Theorem \ref{thm:poly-sparse}.

We use the following standard recursive model.  Start with $G_1$ a single vertex and $\mathcal S(G_1)=\{V(G_1)\}$.  Given $(G_i,\mathcal S(G_i))$, form $G_{i+1}$ by taking a master copy $H$ of $G_i$; for every $S\in\mathcal S(H)$ take a copy $H_S$ of $G_i$; and for every $X\in\mathcal S(H_S)$ add a vertex $v_{S,X}$ adjacent to all vertices of $X$ and to no other vertices.  The new stable-set family consists of all sets $S\cup X$ and $S\cup\{v_{S,X}\}$ obtained in this way.  This is the formulation used by Felsner, Joret, Micek, Trotter and Wiechert.  Since every new vertex is joined to a stable set and no edges are added inside that stable set, induction shows that all graphs in this standard sequence are triangle-free.  We shall use the following standard Burling invariant: for every $j$, every proper colouring of $G_j$ has a set $S\in\mathcal S(G_j)$ on which at least $j$ colours appear.  The assertion is trivial for $G_1$.  Suppose it holds for $G_j$, and colour $G_{j+1}$.  Apply the invariant to the master copy $H$ and choose $S\in\mathcal S(H)$ receiving at least $j$ colours.  Apply the invariant to the copy $H_S$ and choose $X\in\mathcal S(H_S)$ receiving at least $j$ colours.  If the colour set on $X$ is not contained in the colour set on $S$, then $S\cup X\in\mathcal S(G_{j+1})$ receives at least $j+1$ colours.  If the two colour sets are equal, then the new vertex $v_{S,X}$ adjacent to all of $X$ must receive a colour outside this common set, so $S\cup\{v_{S,X}\}\in\mathcal S(G_{j+1})$ receives at least $j+1$ colours.  Thus the invariant holds for $G_{j+1}$, and in particular $\chi(G_j)\ge j$.  The matching upper bound follows by colouring every old copy of $G_{j-1}$ with colours $1,\ldots,j-1$ and giving every vertex added at the last step colour $j$.  Hence $\chi(G_j)=j$.  We use the normalization $B_j=G_j$ for the standard sequence.

\begin{proposition}[Size growth in the standard Burling recursion]\label{prop:burling-growth}
Let $G_j$ be the recursive Burling graphs defined with their stable-set collections $\mathcal S(G_j)$ as in Felsner, Joret, Micek, Trotter and Wiechert \cite{FelsnerJoretMicekTrotterWiechert2018}.  Then
\[
        |\mathcal S(G_j)|=2^{2^{j-1}-1},
\]
and, for $j\ge2$,
\[
        e(G_j)\ge 2^{2^{j-1}-2}.
\]
In particular, for the usual Burling sequence $B_j$ with $\chi(B_j)=j$, the number of edges is at least
\[
        \exp(\Omega(2^{\chi(B_j)})),
\]
and hence is super-polynomial in the chromatic parameter.
\end{proposition}

\begin{proof}
The construction starts with one stable set, so $s_1:=|\mathcal S(G_1)|=1$.  At step $j$, the construction forms two stable sets $S\cup X$ and $S\cup\{v_{S,X}\}$ for every pair $S\in\mathcal S(H)$ and $X\in\mathcal S(H_S)$, where both $H$ and $H_S$ are copies of $G_{j-1}$.  These $2s_{j-1}^2$ stable sets are distinct.  Sets of the form $S\cup X$ contain vertices in the copy $H_S$ and no new vertex $v_{S,X}$, whereas sets of the form $S\cup\{v_{S,X}\}$ contain the unique new vertex $v_{S,X}$.  Distinct choices of $S$ use distinct master-copy parts, and, for fixed $S$, distinct choices of $X$ are distinct by induction.  Hence
\[
        s_j=2s_{j-1}^2.
\]
Solving this recurrence gives $s_j=2^{2^{j-1}-1}$.

For every pair $(S,X)$ the construction adds a vertex $v_{S,X}$ adjacent to all vertices of the nonempty stable set $X$, and hence adds at least one edge.  Therefore
\[
        e(G_j)\ge s_{j-1}^2=2^{2^{j-1}-2}.
\]
The final assertion follows from the standard normalization $\chi(B_j)=j$ used for Burling's sequence; see, for example, the formulation in Pettie--Tardos--Walczak \cite{PTW2026}.
\end{proof}

The fractional analogue has a different bottleneck.  The integer proof loses a factor depending on $|V|$ because it must union-bound over all ordinary colourings.  Fractionally, the corresponding random-subgraph preservation theorem is already known.  The missing issue is instead how expensive it is, in fractional chromatic number, to remove the edges responsible for short cycles.

\begin{definition}[Fractional short-cycle cost]
For a graph $X$ and an integer $r\ge4$, define
\[
        \rho_{<r}(X)=\min\{\chif(V(X),F):F\subseteq E(X),\ X-F\text{ has girth at least }r\}.
\]
Equivalently, $\rho_{<r}(X)$ is the minimum fractional chromatic number of an edge set whose deletion hits every cycle of length less than $r$.
\end{definition}

\begin{proposition}[Fractional product deletion]\label{prop:fractional-deletion}
Let $G$ be a non-empty graph and let $F\subseteq E(G)$.  Then
\[
        \chif(G)\le \chif(G-F)\,\chif(V(G),F).
\]
Consequently,
\[
        h^{\mathrm f}_r(G)
        \ge
        \frac{\chif(G)}{\rho_{<r}(G)}.
\]
Here the quotient is well-defined because every admissible deletion set has
$\chif(V(G),F)\ge1$, and hence $\rho_{<r}(G)\ge1$.  For the empty graph the assertions are interpreted as trivial.
In particular, if $\tau_{<r}(G)$ is the minimum number of edges meeting all cycles of length less than $r$, then
\[
        h^{\mathrm f}_r(G)
        \ge
        \frac{\chif(G)}{1+\sqrt{2\tau_{<r}(G)}}.
\]
\end{proposition}

\begin{proof}
The first inequality is the standard product inequality for fractional colourings.  If $G_1$ and $G_2$ are two graphs on the same vertex set, then intersections $I_1\cap I_2$, where $I_i$ is independent in $G_i$, are independent in $G_1\cup G_2$.  Multiplying fractional colourings of $G_1$ and $G_2$ gives
\[
        \chif(G_1\cup G_2)\le \chif(G_1)\chif(G_2).
\]
Apply this with $G_1=G-F$ and $G_2=(V(G),F)$.

If $F$ is chosen so that $G-F$ has girth at least $r$, then
\[
        h_r^{\mathrm f}(G)\ge \chif(G-F)
        \ge \frac{\chif(G)}{\chif(V(G),F)}.
\]
Taking the minimum over all such $F$ gives the bound with $\rho_{<r}(G)$.

Finally, if $|F|=s$, then $\chif(V(G),F)\le \chi(V(G),F)\le 1+\sqrt{2s}$, because every $c$-critical graph has at least $\binom c2$ edges.  Applying this to a minimum short-cycle transversal gives the last inequality.
\end{proof}

We next record the small-retention fractional random-subgraph estimate that will be used below.  This is not claimed as a new preservation theorem: it is the principal-set proof of Mohar and Wu \cite{MoharWuRandomFractional}, written in the edge-retention convention and with the dependence on the retention probability displayed.  This is the point at which the fractional theory differs sharply from the integral one: the lower bound below depends on $\chif(G)$ and $p$, but not on $|V(G)|$.

\begin{lemma}[Principal-set sparse majorization]\label{lem:principal-majorization}
Let $G$ be a graph with a nonnegative vertex weight $w$, and order the vertices as
\[
        v_1,v_2,\ldots
\]
in nonincreasing order of weight.  Let $s\ge1$ be an integer.  A set $X$ of size $a$ is called $s$-principal if
\[
        X\subseteq \{v_1,\ldots,v_{sa}\},
\]
where the right side is interpreted as $V(G)$ if $sa>|V(G)|$.  Suppose that $Y\subseteq V(G)$ has no $s$-principal subset of average degree at least $x$ in $G[Y]$, and suppose that every independent subset of $Y$ in $G$ has weight at most $1$.  If $w(V(G))=t$, then
\[
        w(Y)\le 2(\lfloor x\rfloor+1)+\frac{2t}{s}.
\]
\end{lemma}

\begin{proof}
For $v\in Y$, let $d^>(v)$ be the number of neighbours of $v$ in $G[Y]$ that precede $v$ in the global weight order, and put
\[
        L=\{v\in Y:d^>(v)\ge x\}.
\]
Write $L=\{u_1,\ldots,u_m\}$ in the global order.  For each $j$, the prefix
\[
        Y_j=Y\cap\{v_1,\ldots,u_j\}
\]
spans at least $xj$ edges, because each of $u_1,
\ldots,u_j$ has at least $x$ earlier neighbours in $Y_j$.  If $Y_j$ had at most $2j$ vertices, its average degree would be at least $x$.  Since $Y$ has no dense $s$-principal subset, one of two alternatives must hold: either $u_j$ is not among the first $2j$ vertices of $Y$, or the full global prefix $\{v_1,\ldots,u_j\}$ has more than $s|Y_j|$ vertices.  Split $L=L_1\cup L_2$ according to these alternatives.

Let $z_1,z_2,\ldots$ be the vertices of $Y$ in the global order.  Let $u_{j_1},u_{j_2},\ldots$ be the vertices of $L_1$ in the global order.  Since $u_{j_a}\in L_1$, at least $2j_a\ge 2a$ vertices of $Y$ occur no later than $u_{j_a}$; hence $w(u_{j_a})\le w(z_{2a})$.  Pairing consecutive terms of the nonincreasing sequence $w(z_i)$ gives
\[
        w(L_1)\le \sum_{a\ge1} w(z_{2a})\le \frac12 w(Y).
\]
Similarly, let $u_{j'_1},u_{j'_2},\ldots$ be the vertices of $L_2$ in the global order.  Since $u_{j'_a}\in L_2$, the full global prefix ending at $u_{j'_a}$ has more than $s|Y_{j'_a}|$ vertices; in particular at least $s j'_a\ge s a$ vertices of $G$ occur no later than $u_{j'_a}$.  Hence $w(u_{j'_a})\le w(v_{sa})$, and partitioning the global nonincreasing sequence into consecutive blocks of length $s$ gives
\[
        w(L_2)\le \sum_{a\ge1}w(v_{sa})\le \frac1s w(V(G))=\frac ts.
\]
Thus, with $S=Y\setminus L$,
\[
        w(S)\ge \frac12w(Y)-\frac ts.
\]
The graph $G[S]$ is $\lfloor x\rfloor$-degenerate, because every nonempty induced subgraph has a last vertex in the global order with fewer than $x$ earlier neighbours inside that subgraph.  Hence $G[S]$ is $(\lfloor x\rfloor+1)$-colourable.  Some independent colour class of $G[S]$ has weight at least $w(S)/(\lfloor x\rfloor+1)$, and by hypothesis this weight is at most $1$.  Therefore
\[
        1\ge \frac{\frac12w(Y)-t/s}{\lfloor x\rfloor+1},
\]
which rearranges to the claimed bound.
\end{proof}

\begin{theorem}[Mohar--Wu preservation, edge-retention form]\label{thm:mohar-wu-random}
Let $G$ be a graph with $\chif(G)=t\ge2$, and let $G_p$ be obtained from $G$ by retaining each edge independently with probability $p\in(0,1)$.  Put
\[
        \lambda=-\log(1-p).
\]
For every $c>\log 2$, with probability at least
\[
        1-\frac{e^{-c}}{1-e^{-c}},
\]
one has
\[
        \chif(G_p)
        \ge
        \frac{t}{4(\log(2et)+c)/\lambda+4}.
\]
Consequently,
\[
        \Prob\left(
        \chif(G_p)
        \ge
        \frac{t}{8\log_{1/(1-p)}(2et)+4}
        \right)
        > 1-\frac1{2t}.
\]
In particular, for $0<p\le1/2$,
\[
        \chif(G_p)
        \ge
        \frac{p t}{12\log(16et^2)}
\]
with probability at least $1-1/(2t)$.
\end{theorem}

\begin{proof}
We give the proof because the small-$p$ form is central here.  By LP duality for fractional colouring, choose a nonnegative weight function $w:V(G)\to\mathbb R_{\ge0}$ such that
\[
        w(V(G))=t,
        \qquad
        w(I)\le1
\]
for every independent set $I$ of $G$.  Order the vertices as $v_1,\ldots,v_n$ so that $w(v_i)\ge w(v_{i+1})$.  For an integer $s\ge1$, a nonempty set $X$ of size $a$ is called $s$-principal if
\[
        X\subseteq \{v_1,\ldots,v_{sa}\}.
\]
If $sa>n$, the right side is interpreted as $V(G)$.  The number of $s$-principal sets of size $a$ is at most
\[
        \binom{sa}{a}\le (es)^a.
\]

Let
\[
        s=\lceil t\rceil,
        \qquad
        x=2\frac{\log(es)+c}{\lambda}.
\]
Then $t\le s\le2t$, so $\log(es)\le\log(2et)$.  If an $s$-principal set $X$ has size $a$ and average degree at least $x$ in $G$, then $e_G(X)\ge xa/2$, and the probability that $X$ is independent in $G_p$ is at most
\[
        (1-p)^{xa/2}
        =\exp(-\lambda xa/2)
        =\exp(-a(\log(es)+c)).
\]
Thus the probability that some $s$-principal $a$-set of average degree at least $x$ is independent in $G_p$ is at most $e^{-ca}$.  Since $c>\log 2$, the geometric sum below is finite.  Summing over $a\ge1$, with probability at least
\[
        1-\sum_{a\ge1}e^{-ca}
        =1-\frac{e^{-c}}{1-e^{-c}},
\]
no $s$-principal set of average degree at least $x$ is independent in $G_p$.

Assume this good event holds, and let $A$ be any independent set of $G_p$.  Since $A$ is independent in $G_p$, no dense $s$-principal subset of $A$ exists on the good event.  Also every independent subset of $A$ in $G$ has weight at most $1$.  Applying Lemma \ref{lem:principal-majorization} with $Y=A$, and using $s\ge t$, we obtain
\[
        w(A)\le 2(\lfloor x\rfloor+1)+2\le 2x+4.
\]
Thus $w/(2x+4)$ assigns weight at most $1$ to every independent set of $G_p$.  By fractional-colouring duality,
\[
        \chif(G_p)\ge \frac{w(V(G))}{2x+4}
        \ge \frac{t}{4(\log(2et)+c)/\lambda+4}.
\]
Taking $c=\log(2et)$ gives the displayed probability bound, since
\[
        \frac{e^{-c}}{1-e^{-c}}=\frac1{2et-1}<\frac1{2t}.
\]
Finally, if $0<p\le1/2$, then $\lambda\ge p$ and $\log(16et^2)\ge2\log(2et)$ for $t\ge2$, so the denominator $8\log(2et)/\lambda+4$ is at most $12\log(16et^2)/p$.  This gives the last assertion.
\end{proof}

\begin{theorem}[Fractional random-surgery criterion]\label{thm:fractional-surgery}
Let $r\ge4$, let $G$ satisfy $\chif(G)=t\ge2$, and let $0<p<1$.  Suppose that for some $R\ge1$,
\[
        \Prob\bigl(\rho_{<r}(G_p)\le R\bigr)>\frac1{2t}.
\]
Then
\[
        h_r^{\mathrm f}(G)
        \ge
        \frac{t}{R\bigl(8\log_{1/(1-p)}(2et)+4\bigr)}.
\]
In particular, if $0<p\le1/2$, then
\[
        h_r^{\mathrm f}(G)
        \ge
        c\frac{p t}{R\log(2et)}
\]
for an absolute constant $c>0$.
\end{theorem}

\begin{proof}
By Theorem \ref{thm:mohar-wu-random}, the event
\[
        \chif(G_p)
        \ge
        \frac{t}{8\log_{1/(1-p)}(2et)+4}
\]
has probability at least $1-1/(2t)$.  By hypothesis, the event $\rho_{<r}(G_p)\le R$ has probability greater than $1/(2t)$.  Therefore the two events have a common outcome $X\subseteq G$.  Proposition \ref{prop:fractional-deletion} applied to $X$ gives
\[
        h_r^{\mathrm f}(G)
        \ge h_r^{\mathrm f}(X)
        \ge \frac{\chif(X)}{\rho_{<r}(X)}
        \ge
        \frac{t}{R\bigl(8\log_{1/(1-p)}(2et)+4\bigr)}.
\]
The small-$p$ form follows as in Theorem \ref{thm:mohar-wu-random}.
\end{proof}

\begin{theorem}[Fractional random extraction criterion]\label{thm:fractional-random-extraction}
Let $r\ge4$, let $G$ be a graph with $t=\chif(G)\ge2$, and let $0<\rho<1$.  Put
\[
        L_\rho(t)=\frac{t}{8\log_{1/(1-\rho)}(2et)+4}
\]
and
\[
        S_\rho(G)=\sum_{\ell=3}^{r-1} C_\ell(G)\rho^\ell.
\]
Then
\[
        h^{\mathrm f}_r(G)
        \ge
        \frac{L_\rho(t)}{1+2\sqrt{S_\rho(G)}}.
\]
In particular, if
\[
        1+2\sqrt{S_\rho(G)}<\frac{L_\rho(t)}q,
\]
equivalently if $L_\rho(t)>q$ and
\[
        S_\rho(G)<\frac14\left(\frac{L_\rho(t)}q-1\right)^2,
\]
then $G$ contains a subgraph of girth at least $r$ and fractional chromatic number greater than $q$.
\end{theorem}

\begin{proof}[Proof of Theorem \ref{thm:fractional-random-extraction}]
Let $X=G_\rho$ be the random retained-edge subgraph.  By Theorem \ref{thm:mohar-wu-random},
\[
        \chif(X)\ge L_\rho(t)
\]
holds with probability at least $1-1/(2t)$.  Let $N$ be the number of cycles of $X$ of length less than $r$.  Then
\[
        \E N=S_\rho(G).
\]
If $S_\rho(G)=0$, then $N=0$ almost surely.  If $S_\rho(G)>0$, Markov's inequality gives
\[
        \Pr\bigl(N\le 2S_\rho(G)\bigr)\ge \frac12.
\]
Thus in either case there is an outcome for which the Mohar--Wu lower bound holds and $N\le 2S_\rho(G)$; when $S_\rho(G)=0$ this means $N=0$.  Choose one edge from each short cycle of $X$ and call the chosen set $F$.  Then $|F|\le 2S_\rho(G)$ and $X-F$ has girth at least $r$.  Proposition \ref{prop:fractional-deletion}, applied inside $X$, gives
\[
        h_r^{\mathrm f}(G)
        \ge h_r^{\mathrm f}(X)
        \ge \frac{\chif(X)}{\chif(V(X),F)}.
\]
Since a graph with $s$ edges has chromatic number, and hence fractional chromatic number, at most $1+\sqrt{2s}$, we have
\[
        \chif(V(X),F)\le 1+\sqrt{2|F|}
        \le 1+2\sqrt{S_\rho(G)}.
\]
This proves the first displayed bound.  The final assertion follows from the sign-sensitive condition $1+2\sqrt{S_\rho(G)}<L_\rho(t)/q$.
\end{proof}

\begin{corollary}[A precise reduction for the fractional conjecture]\label{cor:fractional-reduction}
Fix $r\ge4$.  Suppose that for every $K>0$ there is $T$ such that every graph $G$ with $t=\chif(G)\ge T$ admits a number $p=p(G)\in(0,1)$ for which
\[
        \Prob\left(
        \rho_{<r}(G_p)
        \le
        \frac{t}{K\bigl(8\log_{1/(1-p)}(2et)+4\bigr)}
        \right)>\frac1{2t}.
\]
Then every graph of sufficiently large fractional chromatic number contains a subgraph of girth at least $r$ and fractional chromatic number at least $K$.
\end{corollary}

\begin{proof}
Let
\[
        R=\frac{t}{K\bigl(8\log_{1/(1-p)}(2et)+4\bigr)}.
\]
If $R<1$, then the event $\rho_{<r}(G_p)\le R$ is empty for non-empty $G$, because $\rho_{<r}(X)\ge1$ for every non-empty graph $X$.  This contradicts the displayed hypothesis.  Hence $R\ge1$, and Theorem \ref{thm:fractional-surgery} applies with this value of $R$.  The conclusion gives $h_r^{\mathrm f}(G)\ge K$.
\end{proof}

\begin{remark}[What remains open fractionally]\label{rem:fractional-gap}
Mohar and Wu proved the fractional Erd\H{o}s--Hajnal statement for $r=4$ and proposed the general-girth fractional conjecture \cite{MoharWuTriangleFreeFractional}.  They also proved the conjectured conclusion for Kneser graphs of sufficiently large fractional chromatic number \cite{MoharWuKneser}.  Theorem \ref{thm:mohar-wu-random} shows that small-$p$ fractional preservation is already available with no dependence on $|V(G)|$.  Therefore, within this random-surgery framework, the remaining obstacle is not preservation of fractional chromatic number under thinning; it is a bound on the fractional short-cycle cost $\rho_{<r}(G_p)$ for a suitably chosen sparse random subgraph.
\end{remark}

\begin{remark}[Why the deterministic deletion bound is not enough]
The estimate using $\tau_{<r}$ in Proposition \ref{prop:fractional-deletion} is often much too crude.  For instance, when $G=K_t$ and $r=4$, any triangle-free spanning subgraph $Y$ of $K_t$ has $\omega(Y)\le2$, so the complement $K_t-E(Y)$ has independence number at most $2$ and hence fractional chromatic number at least $t/2$.  Yet classical random constructions contain triangle-free subgraphs of $K_t$ with fractional chromatic number tending to infinity.  Thus a proof of the full fractional Erd\H{o}s--Hajnal conjecture must exploit the structure of the short cycles in the random subgraph, not merely count them or delete an arbitrary transversal.
\end{remark}

\subsection{Cheap cycle killing and a complete-graph proof of concept}\label{subsec:cheap-cycle-killing}

We now isolate one concrete way to make the fractional surgery criterion effective.  For a graph $X$ define the local short-cycle multiplicity
\[
        \lambda_{<r}(X)=\max_{v\in V(X)}
        \#\{C: C\text{ is a cycle of }X,\ 3\le |C|<r,\ v\in V(C)\}.
\]
Low local multiplicity gives a cheap transversal in the fractional sense.

\begin{lemma}[Local spread gives cheap cycle killing]\label{lem:local-spread-cheap}
For every graph $X$,
\[
        \rho_{<r}(X)\le \lambda_{<r}(X)+1.
\]
In particular, if the cycles of length less than $r$ in $X$ are pairwise vertex-disjoint, then $\rho_{<r}(X)\le2$.
\end{lemma}

\begin{proof}
Choose one edge from each cycle of length less than $r$, and let $F$ be the set of chosen edges.  Then $X-F$ has girth at least $r$.  A vertex $v$ can be incident with a chosen edge only if that edge was chosen from a short cycle containing $v$.  Since we choose only one edge per cycle, each short cycle containing $v$ contributes at most one selected edge incident with $v$.  Hence
\[
        \Delta(V(X),F)\le \lambda_{<r}(X).
\]
Therefore
\[
        \chif(V(X),F)\le \chi(V(X),F)\le \lambda_{<r}(X)+1,
\]
which proves the first claim.  If the short cycles are pairwise vertex-disjoint, choose one edge from each of them.  The chosen edges are then pairwise vertex-disjoint, so $(V(X),F)$ is a matching and has fractional chromatic number at most $2$.
\end{proof}

Combining this deterministic observation with Mohar--Wu preservation gives the following spread version of the fractional random-surgery criterion.

\begin{proposition}[Fractional spread-thinning criterion]\label{prop:fractional-spread-criterion}
Let $G$ be a graph with $t=\chif(G)\ge2$, and let $p\in(0,1)$.  Suppose that for some $D\ge0$,
\[
        \Prob\bigl(\lambda_{<r}(G_p)\le D\bigr)>\frac1{2t}.
\]
Then
\[
        h_r^{\mathrm f}(G)
        \ge
        \frac{t}{(D+1)\bigl(8\log_{1/(1-p)}(2et)+4\bigr)}.
\]
If, with probability greater than $1/(2t)$, the short cycles of $G_p$ are pairwise vertex-disjoint, then the denominator may be replaced by
\[
        2\bigl(8\log_{1/(1-p)}(2et)+4\bigr).
\]
\end{proposition}

\begin{proof}
The event in Theorem \ref{thm:mohar-wu-random}, namely
\[
        \chif(G_p)
        \ge
        \frac{t}{8\log_{1/(1-p)}(2et)+4},
\]
has probability at least $1-1/(2t)$.  By the hypothesis, it has a common outcome with the event $\lambda_{<r}(G_p)\le D$.  For that outcome $X$, Lemma \ref{lem:local-spread-cheap} gives $\rho_{<r}(X)\le D+1$, and Proposition \ref{prop:fractional-deletion} gives
\[
        h_r^{\mathrm f}(G)
        \ge
        h_r^{\mathrm f}(X)
        \ge
        \frac{\chif(X)}{D+1}.
\]
The pairwise vertex-disjoint case is identical, using $\rho_{<r}(X)\le2$.
\end{proof}

The next proposition verifies this mechanism in the densest possible ambient graph.  It is not needed for the integer sparse theorem; it is included to show that the fractional obstruction is genuinely a spread problem for short cycles, rather than a failure of small-retention fractional preservation.

\begin{proposition}[Cheap cycle killing in random thinnings of complete graphs]\label{prop:complete-graph-cheap}
Fix $r\ge4$ and choose
\[
        0<\delta<\frac1{2(r-1)}.
\]
Then there is a constant $c=c(r,\delta)>0$ such that, for all sufficiently large $n$, the complete graph $K_n$ contains a subgraph $Y$ and a set $F\subseteq E(Y)$ such that
\[
        \chif(Y)\ge c\frac{n^\delta}{\log n},
        \qquad
        \Delta(V(Y),F)\le1,
        \qquad
        \girth(Y-F)\ge r.
\]
Consequently
\[
        h_r^{\mathrm f}(K_n)
        \ge
        c\frac{n^\delta}{\log n}.
\]
\end{proposition}

\begin{proof}
Let $p=n^{-1+\delta}$ and let $Y=(K_n)_p$.  By Theorem \ref{thm:mohar-wu-random}, with probability at least $1-1/(2n)$,
\[
        \chif(Y)
        \ge
        \frac{n}{8\log_{1/(1-p)}(2en)+4}.
\]
For all large $n$ we have $p<1/2$ and $-\log(1-p)\ge p$, so
\[
        \log_{1/(1-p)}(2en)
        =\frac{\log(2en)}{-\log(1-p)}
        \le \frac{2\log n}{p}.
\]
Thus, after adjusting constants,
\[
        \chif(Y)\ge c_1\frac{pn}{\log n}
        = c_1\frac{n^\delta}{\log n}
\]
with probability at least $1-1/(2n)$.

It remains to show that, with probability tending to $1$, no two cycles of length less than $r$ in $Y$ share a vertex.  Consider two distinct cycles $C,C'$ of lengths between $3$ and $r-1$ whose union is connected, equivalently whose vertex sets intersect.  If their union has $v$ vertices and $e$ edges, then it has cyclomatic number at least $2$, and hence $e\ge v+1$.  Also $v\le |C|+|C'|-1\le 2(r-1)-1$.  For each fixed isomorphism type of such a union, the expected number of labelled copies in $Y$ is at most
\[
        O_r(n^v p^e)
        \le O_r(n^v p^{v+1})
        = O_r\bigl(n^{-1+\delta(v+1)}\bigr)
        \le O_r\bigl(n^{-1+2\delta(r-1)}\bigr)=o(1).
\]
There are only finitely many relevant isomorphism types, depending on $r$.  Hence the expected number of intersecting pairs of short cycles is $o(1)$, and Markov's inequality shows that, with probability tending to $1$, no such pair exists.

For a common outcome of the two events above, choose one edge from each cycle of length less than $r$ and let $F$ be the chosen set.  The short cycles are pairwise vertex-disjoint, so $F$ is a matching.  Deleting $F$ kills every cycle of length less than $r$, and therefore $Y-F$ has girth at least $r$.  Since $\chif(V(Y),F)\le2$, Proposition \ref{prop:fractional-deletion} gives
\[
        \chif(Y-F)\ge \frac{\chif(Y)}2
        \ge \frac{c_1}{2}\frac{n^\delta}{\log n}.
\]
This proves the proposition.
\end{proof}

\subsection{Clique-organised cheap cycle killing}\label{subsec:clique-organised}

The complete graph and the projective-plane line-graph example below are instances of a common mechanism.  Short cycles are organised inside a controlled family of cliques, and after thinning each clique is locally sparse enough that no vertex is contained in too many surviving short cycles.  We record the general form because it gives a reusable test for other structured families.

\begin{definition}[Clique covers of short cycles]
Let $r\ge4$.  A family $\mathcal K$ of cliques of a graph $G$ is a \emph{$(<r)$-cycle clique cover} if every cycle of $G$ of length less than $r$ is contained in some member of $\mathcal K$.  Its vertex-incidence number and local overlap are
\[
        I(\mathcal K)=\sum_{K\in\mathcal K}|V(K)|,
        \qquad
        b(\mathcal K)=\max_{v\in V(G)}|\{K\in\mathcal K:v\in V(K)\}|.
\]
For a clique $K\in\mathcal K$, a vertex $v\in K$, and a retention probability $p$, let $Z_{K,v}$ be the number of cycles of length less than $r$ in $K[p]$ that contain $v$.
\end{definition}

\begin{theorem}[Clique-covered cheap cycle killing]\label{thm:clique-covered-cheap}
Let $G$ be a graph with $t=\chif(G)\ge2$, let $\mathcal K$ be a $(<r)$-cycle clique cover of $G$, and let $p\in(0,1)$.  Suppose that $b=b(\mathcal K)<\infty$ and that, for some $D\ge0$ and $0<\xi<1/(2t)$,
\[
        \sum_{K\in\mathcal K}\sum_{v\in V(K)}
        \Pr(Z_{K,v}>D)
        \le \xi .
\]
Then
\[
        h_r^{\mathrm f}(G)
        \ge
        \frac{t}{(bD+1)\bigl(8\log_{1/(1-p)}(2et)+4\bigr)}.
\]
In particular, along any sequence of such triples $(G,p,D)$ with $t=\chif(G)\to\infty$, if
\[
        bD+1=o\left(\frac{pt}{\log t}\right),
\]
then $h_r^{\mathrm f}(G)\to\infty$.
\end{theorem}

\begin{proof}
By the union bound, with probability at least $1-\xi$ every vertex-clique occurrence $(K,v)$ has at most $D$ cycles of length less than $r$ through $v$ inside $K[p]$.  On this event, every short cycle of $G_p$ is contained in some $K[p]$, and each vertex lies in at most $b$ cliques of $\mathcal K$.  Hence
\[
        \lambda_{<r}(G_p)\le bD.
\]
Since $\xi<1/(2t)$, this event has probability greater than $1/(2t)$.  Proposition \ref{prop:fractional-spread-criterion} gives the displayed bound.
\end{proof}

The next lemma shows that polynomially many clique incidences are harmless when each clique is thinned to a genuinely sparse random graph.  It is deliberately stated with a non-optimized constant $D$; the important point is that $D$ is independent of the clique size.

\begin{lemma}[Ambient random-clique tail bound]\label{lem:random-clique-tail}
Fix $r\ge4$, $A>0$, and $0<\delta<1/(r-1)$.  There are an integer $D=D(r,A,\delta)$ and a constant $N_0=N_0(r,A,\delta)$ such that the following holds.  Let $N\ge N_0$, let $K$ be any clique of size at most $N$, let $p\le N^{-1+\delta}$, and let $X=K[p]$.  For every fixed vertex $v$ of $K$,
\[
        \Pr\bigl(v\text{ lies on more than }D\text{ cycles of }X\text{ of length }<r\bigr)
        \le N^{-A}.
\]
\end{lemma}

\begin{proof}
The point of the statement is that the tail is measured in the ambient parameter $N$, not in the actual clique size.  Fix $D$ for the moment.  If $v$ lies on more than $D$ short cycles, choose $D+1$ distinct short cycles through $v$ and let $F$ be the union of their edges and non-isolated vertices, rooted at $v$.  Then $F$ is connected, every cycle used in its construction has length less than $r$, and $F$ contains at least $D+1$ distinct simple cycles through the root.  There are only finitely many rooted isomorphism types of such unions, depending on $r$ and $D$.  For such an $F$, write $v(F)$ and $e(F)$ for its numbers of vertices and edges, and put
\[
        c(F)=e(F)-v(F)+1.
\]
The binary cycle space of a graph of cyclomatic number $c$ has dimension $c$.  Each simple cycle gives a distinct nonzero vector in this space, so the graph has at most $2^c-1$ distinct simple cycles.  Thus the presence of $D+1$ rooted cycles forces
\[
        c(F)\ge \log_2(D+2).
\]
For the edge count, choose a minimal subfamily of the rooted cycles whose union is $F$.  Add these cycles one by one.  Each newly added cycle contains at least one edge that was not present in the previous union; since it is itself a cycle, adding it increases the cyclomatic number of the current union by at least one.  If $m$ cycles are added, then $m\le c(F)$, and because each added cycle has length at most $r-1$,
\[
        e(F)\le (r-1)m\le (r-1)c(F).
\]

If the clique has size $K\le N$, the expected number of rooted labelled copies of $F$ in $K[p]$ is at most
\[
        C_F K^{v(F)-1}p^{e(F)}
        \le C_F N^{v(F)-1-(1-\delta)e(F)}
        = C_F N^{-c(F)+\delta e(F)}
        \le C_F N^{-(1-\delta(r-1))c(F)}.
\]
Choose $D$ so large that
\[
        (1-\delta(r-1))\log_2(D+2)>A+1.
\]
After summing over the finite set of rooted union types, the probability is at most $N^{-A}$ for all sufficiently large $N$.
\end{proof}

\begin{corollary}[Polynomial clique covers]\label{cor:polynomial-clique-cover}
Fix $r\ge4$, $a>0$, $b\ge1$, and $0<\delta<1/(r-1)$.  Let $(G_N)$ be graphs with $t_N=\chif(G_N)\to\infty$.  Suppose each $G_N$ has a $(<r)$-cycle clique cover $\mathcal K_N$ such that, for all sufficiently large $N$,
\[
        b(\mathcal K_N)\le b,
        \qquad
        I(\mathcal K_N)\le N^a,
        \qquad
        |K|\le N\quad(K\in\mathcal K_N).
\]
Assume additionally that the subgraph induced by the vertices not lying in any clique of $\mathcal K_N$ has fractional chromatic number $o(t_N)$.  If
\[
        \frac{N^{-1+\delta}t_N}{\log t_N}\to\infty,
\]
then
\[
        h_r^{\mathrm f}(G_N)\to\infty.
\]
More quantitatively, for $p_N=N^{-1+\delta}$,
\[
        h_r^{\mathrm f}(G_N)
        \ge
        c\frac{p_Nt_N}{\log t_N}
\]
for all large $N$, where $c>0$ depends only on $r,a,b,\delta$.
\end{corollary}

\begin{proof}
Let $U_N$ be the union of the cliques in $\mathcal K_N$ and let $W_N=V(G_N)\setminus U_N$.  Fractional chromatic number is subadditive over vertex partitions: combining fractional colourings of $G_N[U_N]$ and $G_N[W_N]$ gives
\[
        \chif(G_N)\le \chif(G_N[U_N])+\chif(G_N[W_N]).
\]
Put
\[
        t'_N=\chif(G_N[U_N]).
\]
The assumption $\chif(G_N[W_N])=o(t_N)$ implies $t'_N=(1-o(1))t_N$.  We now apply Theorem~\ref{thm:clique-covered-cheap} to $G_N[U_N]$, whose fractional chromatic number is $t'_N$.  Since $t'_N\asymp t_N$, the hypotheses
\[
        \frac{N^{-1+\delta}t_N}{\log t_N}\to\infty,
        \qquad
        t_N\le N^a
\]
remain valid after replacing $t_N$ by $t'_N$, up to changing constants.  Also every vertex of $G_N[U_N]$ lies in at least one clique of the cover, and $t'_N\le |U_N|\le I(\mathcal K_N)\le N^a$.

Choose
\[
        A=2a+3.
\]
Apply Lemma \ref{lem:random-clique-tail} with this value of $A$ to every vertex-clique incidence in the cover, and let $D=D(r,A,\delta)$ be the resulting bound.  Since $I(\mathcal K_N)\le N^a$, the total probability of any bad vertex-clique incidence is at most
\[
        N^a\cdot N^{-A}=N^{-(a+3)}<\frac1{2N^a}\le \frac1{2t'_N}
\]
for all sufficiently large $N$.  Theorem~\ref{thm:clique-covered-cheap}, applied to $G_N[U_N]$, gives
\[
        h_r^{\mathrm f}(G_N)
        \ge h_r^{\mathrm f}(G_N[U_N])
        \ge c\frac{p_Nt'_N}{\log t'_N}
        \ge c'\frac{p_Nt_N}{\log t_N}.
\]
This proves both the quantitative lower bound and the divergence statement.
\end{proof}

\begin{remark}[Scope of the clique-cover theorem]
The complete graph is covered by one clique, so Proposition \ref{prop:complete-graph-cheap} is the extremal one-clique instance.  The projective-plane line graphs in Section \ref{subsec:projective-linegraphs} are covered by the point- and line-cliques, with bounded local overlap and polynomially many incidences.  Kneser graphs, however, are not generally covered by this hypothesis for $r\ge5$: a $4$-cycle in a Kneser graph need not be contained in a clique.  Mohar--Wu's Kneser-graph verification therefore uses additional Kneser-specific structure rather than the clique-cover mechanism isolated here.
\end{remark}

\begin{question}[Cheap cycle killing]\label{q:cheap-cycle-killing-main}
Fix $r\ge5$.  Does there exist a choice of $p=p(t)$ and a function $D(t)\ge0$ with $D(t)+1=o(p(t)t/\log t)$ such that, for every graph $G$ with $\chif(G)=t$,
\[
        \Prob\bigl(\lambda_{<r}(G_p)\le D(t)\bigr)>\frac1{2t}
\]
for all sufficiently large $t$?
\end{question}

A positive answer to Question \ref{q:cheap-cycle-killing-main} would imply the full higher-girth fractional Erd\H{o}s--Hajnal conjecture.  Indeed, Proposition \ref{prop:fractional-spread-criterion} would give
\[
        h_r^{\mathrm f}(G)
        \ge
        \Omega\left(\frac{p(t)t}{(D(t)+1)\log t}\right)\to\infty.
\]
Proposition \ref{prop:complete-graph-cheap} proves this type of dispersion for the complete graph after thinning.  The general problem is whether every high-fractional-chromatic graph admits a useful thinning in which short cycles are similarly spread, or whether some family keeps short cycles locally concentrated at every retention probability that preserves fractional chromatic number.

\section{Applications, obstructions, and quantitative refinements}

Theorem \ref{thm:poly-sparse} rules out counterexamples that contain arbitrarily high-chromatic critical cores with edge count bounded by a fixed power of their chromatic number.  Thus any remaining counterexample sequence must avoid such polynomial-density critical cores.  The lower-bound examples of Pettie, Tardos and Walczak show that tower-type phenomena already occur at $r=5$ \cite{PTW2026}; their Burling-graph framework is therefore a natural test case for any further sharpening.

\subsection{Sparse locally dense test cases: projective-plane line graphs}\label{subsec:projective-linegraphs}

We next examine the opposite extreme from the complete graph.  The incidence graph itself of a projective plane is not a test case for high fractional chromatic number: it is bipartite, and hence has fractional chromatic number $2$.  The relevant locally dense object associated with a projective plane is instead the line graph of its incidence graph.  This graph is sparse globally, but each point and each line of the plane creates a large clique in the line graph.

Let $\Pi_q$ be a projective plane of order $q$ (for example, with $q$ a prime power), let $B_q$ be its point-line incidence graph, and put
\[
        X_q=L(B_q).
\]
The graph $B_q$ is $(q+1)$-regular and bipartite, has $2(q^2+q+1)$ vertices and girth $6$.  The vertices of $X_q$ are the incidences of $\Pi_q$.  Each point and each line of the plane gives a clique of size $q+1$ in $X_q$, and every edge of $X_q$ lies in exactly one of these cliques.

\begin{lemma}[Fractional chromatic number of the line-graph test case]\label{lem:linegraph-chif}
For the graph $X_q=L(B_q)$,
\[
        \chif(X_q)=q+1.
\]
\end{lemma}

\begin{proof}
The fractional chromatic number of a line graph is the fractional edge-chromatic number of the original graph: independent sets in $L(B_q)$ are matchings of $B_q$, so the fractional colouring linear program for $L(B_q)$ is precisely the fractional edge-colouring program for $B_q$.  Since $B_q$ is bipartite, the matching polytope is integral and the fractional edge-chromatic number equals the maximum degree.  Here $\Delta(B_q)=q+1$.
\end{proof}

For girth $5$, all short cycles in $X_q$ are internal to the point- and line-cliques.

\begin{lemma}[Short cycles in $L(B_q)$]\label{lem:linegraph-short-cycles}
Every triangle and every $4$-cycle of $X_q$ is contained in one of the cliques corresponding to a single point or a single line of the projective plane.
\end{lemma}

\begin{proof}
A cycle in a line graph corresponds to a cyclic sequence of edges of the original graph, consecutive members of the sequence being incident.  If such a triangle or $4$-cycle in $L(B_q)$ is not contained in the clique generated by one vertex of $B_q$, then the corresponding edges of $B_q$ contain a cycle of length at most $4$.  This is impossible because $B_q$ has girth $6$.  Hence the line-graph cycle is contained in one star of $B_q$, which is precisely one of the point- or line-cliques of $X_q$.
\end{proof}

\begin{proposition}[Cheap cycle killing for projective-plane line graphs]\label{prop:projective-linegraph-cheap}
Fix $B>1$ and put
\[
        t=q+1,
        \qquad
        p=\frac{(\log t)^B}{t}.
\]
Let $Y=(X_q)_p$ be the random retained-edge subgraph of $X_q$.  Then there is a constant $D_B$ such that, with probability tending to $1$ as $q\to\infty$ through orders for which such a projective plane exists,
\[
        \lambda_{<5}(Y)\le D_B.
\]
Consequently,
\[
        h^{\mathrm f}_5(X_q)
        \ge c_B (\log t)^{B-1}
\]
for all sufficiently large orders $q$ for which such a projective plane exists, where $c_B>0$ depends only on $B$.
\end{proposition}

\begin{proof}
There are $2(q^2+q+1)=O(t^2)$ point- and line-cliques, each of size $t$, and every edge of $X_q$ belongs to exactly one such clique.  By Lemma \ref{lem:linegraph-short-cycles}, every triangle and every $4$-cycle of $X_q$ is contained in one of these cliques.

Choose a fixed number $\delta$ with $0<\delta<1/4$.  Since $p=(\log t)^B/t$, we have $p\le t^{-1+\delta}$ for all sufficiently large $t$.  Apply Lemma~\ref{lem:random-clique-tail} with $r=5$ and $A=6$ to a clique of size $t$ in ambient parameter $N=t$.  It gives a constant $D=D(B)$ only through the fixed choices just made, and for each vertex-clique occurrence the probability of lying on more than $D$ retained triangles or $4$-cycles inside that clique is at most $t^{-6}$.  There are $O(t^3)$ vertex-clique occurrences, so with probability tending to $1$ no occurrence is bad.  Since each vertex of $X_q$ belongs to exactly two of the point/line cliques, on this event
\[
        \lambda_{<5}(Y)\le 2D=:D_B.
\]
This avoids any binomial domination assertion for rooted $4$-cycles; such cycles can cluster through a common intermediate vertex, and the rooted-union bound in Lemma~\ref{lem:random-clique-tail} is designed to handle that clustering.

By Lemma \ref{lem:linegraph-chif}, $\chif(X_q)=t$.  The Mohar--Wu preservation theorem, Theorem \ref{thm:mohar-wu-random}, gives with probability at least $1-o(1)$ that
\[
        \chif(Y)
        \ge
        c\frac{pt}{\log t}
        =c(\log t)^{B-1}.
\]
For a common outcome, Lemma \ref{lem:local-spread-cheap} gives
\[
        \rho_{<5}(Y)\le \lambda_{<5}(Y)+1\le D_B+1.
\]
Proposition \ref{prop:fractional-deletion} then yields a subgraph of $Y$, and hence of $X_q$, of girth at least $5$ and fractional chromatic number at least a constant multiple of $(\log t)^{B-1}$.
\end{proof}

The abstract clique-cover theorem gives a stronger polynomial lower bound for the same family if one uses the larger retention probability $p=t^{-1+\theta}$.

\begin{corollary}[Polynomial fractional growth for projective-plane line graphs]\label{cor:projective-linegraph-polynomial}
Let $X_q=L(B_q)$ and $t=q+1$.  For every $0<\theta<1/4$ there is a constant $c=c(\theta)>0$ such that
\[
        h^{\mathrm f}_5(X_q)
        \ge c\frac{t^\theta}{\log t}
\]
for all sufficiently large orders $q$ for which such a projective plane exists.
\end{corollary}

\begin{proof}
The point- and line-cliques of the projective plane form a $(<5)$-cycle clique cover by Lemma \ref{lem:linegraph-short-cycles}.  Each vertex of $X_q$ lies in exactly two of these cliques, the total incidence is $2(q^2+q+1)(q+1)=O(t^3)$, and the maximum clique size is $t$.  Also $\chif(X_q)=t$ by Lemma \ref{lem:linegraph-chif} and every vertex of $X_q$ lies in the clique cover.  Corollary \ref{cor:polynomial-clique-cover}, with $N=t$, $a=3$, $b=2$, $r=5$, and the given $\theta<1/4$, gives the result.

\end{proof}

\begin{lemma}[Cycles in line graphs]\label{lem:linegraph-cycle-cover}
Let $B$ be a graph and let $C$ be a cycle of length $m$ in the line graph $L(B)$.  If the vertices of $C$, viewed as edges of $B$, are not all incident with one common vertex of $B$, then $B$ contains a cycle of length at most $m$.
\end{lemma}

\begin{proof}
Let the cycle in $L(B)$ have vertices $e_1,\ldots,e_m$, where the $e_i$ are distinct edges of $B$ and $e_i$ meets $e_{i+1}$, indices modulo $m$.  For each $i$, choose $x_i\in e_i\cap e_{i+1}$.  Then
\[
        x_1,e_2,x_2,e_3,\ldots,x_m,e_1,x_1
\]
is a closed alternating walk in $B$, after deleting any immediate repetitions.  If this closed walk has positive length, it contains a cycle; deleting repeated closed subwalks gives a cycle of length at most $m$.

It remains to consider the case in which every such choice gives the trivial closed walk.  Then, for each $i$, the two intersections $e_{i-1}\cap e_i$ and $e_i\cap e_{i+1}$ must be the same vertex; otherwise choosing them differently would give a nontrivial segment through $e_i$.  Going around the cycle, all consecutive intersections are therefore the same vertex.  Hence all edges $e_1,\ldots,e_m$ are incident with one common vertex of $B$, contrary to the hypothesis.  Thus the nontrivial closed-walk case must occur, and $B$ contains a cycle of length at most $m$.
\end{proof}

\begin{proposition}[Line graphs of high-girth regular bipartite graphs]\label{prop:regular-linegraph-clique}
Fix $r\ge4$, $A>0$, and $0<\theta<1/(r-1)$.  Let $(B_N)$ be a sequence of $\Delta_N$-regular bipartite graphs with
\[
        \Delta_N\to\infty,
        \qquad
        \girth(B_N)\ge r,
        \qquad
        |V(B_N)|\le \Delta_N^A.
\]
Put $X_N=L(B_N)$.  Then there is a constant $c=c(r,A,\theta)>0$ such that
\[
        h_r^{\mathrm f}(X_N)
        \ge c\frac{\Delta_N^\theta}{\log \Delta_N}
\]
for all sufficiently large $N$.
\end{proposition}

\begin{proof}
The fractional chromatic number of $X_N=L(B_N)$ is the fractional edge-chromatic number of $B_N$.  Since $B_N$ is bipartite and $\Delta_N$-regular,
\[
        \chif(X_N)=\Delta_N.
\]
For each vertex $v$ of $B_N$, the set of edges of $B_N$ incident with $v$ gives a clique $K_v$ of size $\Delta_N$ in $X_N$.  These star-cliques cover all edges of $X_N$, and every vertex of $X_N$ belongs to exactly two of them.

The family $\{K_v:v\in V(B_N)\}$ is a $(<r)$-cycle clique cover of $X_N$.  Indeed, if a cycle of length less than $r$ in $X_N=L(B_N)$ is not contained in one of these star-cliques, Lemma~\ref{lem:linegraph-cycle-cover} gives a cycle of $B_N$ of no greater length, contradicting $\girth(B_N)\ge r$.

The total clique incidence is
\[
        \sum_{v\in V(B_N)} |K_v|=|V(B_N)|\Delta_N\le \Delta_N^{A+1}.
\]
Apply Corollary \ref{cor:polynomial-clique-cover} with $N=\Delta_N$, clique-incidence exponent $A+1$, vertex-incidence number $b=2$, and the given $\theta<1/(r-1)$.  Every vertex of $X_N$ lies in the clique cover, so the exceptional outside fractional chromatic number is $0$.  This proves the proposition.
\end{proof}

\begin{corollary}[Equal-order generalized polygons as fractional test cases]\label{cor:generalized-polygons-linegraphs}
For every sequence of existing finite generalized quadrangles or generalized hexagons of equal order $(s,s)$ with $s\to\infty$, the line graphs of their incidence graphs satisfy
\[
        h_r^{\mathrm f}(L(B_s))\to\infty
\]
for every fixed $r$ not exceeding the girth of the incidence graph.  More quantitatively, if $\Delta_s=s+1$, then for every $0<\theta<1/(r-1)$,
\[
        h_r^{\mathrm f}(L(B_s))\ge c\frac{\Delta_s^\theta}{\log\Delta_s}.
\]
\end{corollary}

\begin{proof}
For generalized quadrangles and generalized hexagons of order $(s,s)$ the incidence graph is $(s+1)$-regular, has girth $8$ and $12$ respectively, and has polynomially many vertices in $s+1$.  Proposition \ref{prop:regular-linegraph-clique} applies.  Biregular generalized polygons can be treated by the same argument after replacing regularity by the corresponding fractional edge-chromatic number, but the equal-order form is the clean statement used here.
\end{proof}

This proposition verifies cheap cycle killing for a sparse, locally dense family quite different from complete graphs and Kneser graphs.  It also clarifies why the incidence graph itself is not the right stress test: its fractional chromatic number is always $2$, while the line graph has fractional chromatic number $q+1$ and contains large local cliques.

\subsection{A rooted spread criterion for other sparse families}\label{subsec:rooted-spread}

For triangle-free Ramsey-type graphs and for Burling-type constructions, the issue is not the preservation term; Mohar--Wu already gives that.  What must be checked is whether short cycles in the thinned graph remain locally spread.  The following elementary criterion isolates a directly verifiable sufficient condition.

For a graph $G$ and a vertex $v$, let $Z_v=Z_v(G_p)$ be the number of cycles of length less than $r$ in $G_p$ that contain $v$.

\begin{lemma}[Factorial-moment spread test]\label{lem:factorial-spread}
Let $G$ be a graph with $t=\chif(G)\ge2$, and let $0<p<1$.  Suppose that for some integers $D\ge1$ and some $\xi>0$,
\[
        \sum_{v\in V(G)} \mathbb E\binom{Z_v}{D+1}\le \xi.
\]
Then
\[
        \Pr\bigl(\lambda_{<r}(G_p)\le D\bigr)
        \ge 1-\xi.
\]
Consequently, if $\xi<1/(2t)$, then
\[
        h_r^{\mathrm f}(G)
        \ge
        \frac{\chif(G)}{(D+1)\bigl(8\log_{1/(1-p)}(2e\chif(G))+4\bigr)}.
\]
\end{lemma}

\begin{proof}
If $\lambda_{<r}(G_p)>D$, then some vertex $v$ lies on at least $D+1$ short cycles in $G_p$, and hence $\binom{Z_v}{D+1}\ge1$.  Therefore Markov's inequality gives
\[
        \Pr(\lambda_{<r}(G_p)>D)
        \le \sum_v\mathbb E\binom{Z_v}{D+1}
        \le \xi.
\]
The final assertion follows from Proposition \ref{prop:fractional-spread-criterion}.
\end{proof}

For $r=5$ and triangle-free graphs this criterion has a particularly transparent form.  Only $4$-cycles need to be killed.  If $\mathcal Q_v$ denotes the family of $4$-cycles of $G$ containing $v$, then
\[
        \mathbb E\binom{Z_v}{D+1}
        =
        \sum_{\{Q_1,\ldots,Q_{D+1}\}\subseteq\mathcal Q_v}
        p^{|E(Q_1\cup\cdots\cup Q_{D+1})|}.
\]
Thus the test is exactly a rooted codegree or spread condition for the hypergraph of $4$-cycles through each vertex.  Pseudorandom triangle-free Ramsey graphs should be attacked by estimating this displayed sum.  In particular, it is not enough to count $4$-cycles through a vertex; one must also rule out large clusters of $4$-cycles supported on a small edge set.

\begin{corollary}[A usable Ramsey-type sufficient condition]\label{cor:ramsey-spread}
Fix $r=5$.  Let $G_n$ be triangle-free graphs with $t_n=\chif(G_n)\to\infty$.  Suppose there exist $p_n\in(0,1)$ and integers $D_n\ge1$ such that
\[
        \frac{p_n t_n}{\log t_n}\to\infty,
        \qquad
        D_n+1=o\!\left(\frac{p_n t_n}{\log t_n}\right),
\]
while
\[
        \sum_{v\in V(G_n)} \mathbb E\binom{Z_v(G_{n,p_n})}{D_n+1}=o(1/t_n).
\]
Then
\[
        h_5^{\mathrm f}(G_n)\to\infty.
\]
\end{corollary}

\begin{proof}
Apply Lemma \ref{lem:factorial-spread} with $D=D_n$ and combine it with Theorem \ref{thm:mohar-wu-random}.  For every $0<p<1$ we have $-\log(1-p)\ge p$, and hence
\[
        \log_{1/(1-p)}(2et)=\frac{\log(2et)}{-\log(1-p)}\le \frac{\log(2et)}p.
\]
Thus the denominator in the preservation term is $O(\log t_n/p_n)$, while the additional deletion cost is $D_n+1=o(p_nt_n/\log t_n)$.
\end{proof}

This gives a concrete program for triangle-free Ramsey-type graphs: compute the rooted factorial moments of the thinned $4$-cycle hypergraph.  If they satisfy Corollary \ref{cor:ramsey-spread}, cheap cycle killing succeeds.  If they fail at every retention probability with $p_nt_n/\log t_n\to\infty$, then the obstruction is a genuinely clustered rooted $4$-cycle structure rather than mere abundance of $4$-cycles.

\subsection{\texorpdfstring{A polylogarithmic-codegree pseudorandom $r=5$ theorem}{A polylogarithmic-codegree pseudorandom r=5 theorem}}\label{subsec:pseudorandom}

We now carry out the factorial-moment computation in a genuinely non-clique regime.  The goal is to cover triangle-free graphs at Ramsey scale, where
\[
        |V(G)|\approx \chi_f(G)^2\operatorname{polylog}\chi_f(G),
        \qquad
        \Delta(G)\approx \chi_f(G)\operatorname{polylog}\chi_f(G).
\]
The previous maximum-codegree entropy estimate reached only $L=O(\log t/\log\log t)$.  The next lemma uses the special rooted structure of $4$-cycles in triangle-free graphs and improves this beyond the logarithmic scale: any fixed polylogarithmic maximum codegree is harmless.

For a graph $G$, put
\[
        \operatorname{codeg}(G)=\max_{x\ne y}|N_G(x)\cap N_G(y)|.
\]

If $G$ is triangle-free, a $4$-cycle through a root $v$ has the form
\[
        v a u b v,
        \qquad a,b\in N(v),\quad u\notin N[v].
\]
For a family $\mathcal Q$ of rooted $4$-cycles through $v$, let $A(\mathcal Q)\subseteq N(v)$ be the set of neighbours of $v$ used by the cycles, and let $J(\mathcal Q)$ be the bipartite support graph between $A(\mathcal Q)$ and the opposite vertices $U(\mathcal Q)$, containing all support edges away from $v$.

\begin{lemma}[Root-neighbour support moment bound]\label{lem:root-neighbour-moment}
Let $G$ be triangle-free, let $v\in V(G)$, and let
\[
        \Delta=\Delta(G),\qquad L=\operatorname{codeg}(G).
\]
For $0<p<1$, let $Z_v$ be the number of $4$-cycles through $v$ in $G_p$.  For every integer $q\ge1$,
\[
        \mathbb E\binom{Z_v}{q}
        \le
        \sum_{h=2}^{2q}
        \Delta^h (Lh^2)^q
        p^{h+M(h,q)},                                    \tag{\ref{lem:root-neighbour-moment}.1}
\]
where
\[
        M(h,q)=\max\left\{h,\frac{2q}{h-1}\right\}.
\]
\end{lemma}

\begin{proof}
A $q$-set $\mathcal Q$ of rooted $4$-cycles through $v$ uses some set
\[
        A=A(\mathcal Q)\subseteq N(v)
\]
of root-neighbours.  Put $h=|A|$.  Since each rooted $4$-cycle uses two root-neighbours, $2\le h\le 2q$.  There are at most $\Delta^h$ choices for $A$.

For fixed $A$, each possible rooted $4$-cycle is obtained by choosing an unordered pair $\{a,b\}\subseteq A$ and then choosing a common neighbour $u\in N(a)\cap N(b)$.  There are at most $L\binom h2\le Lh^2$ such cycles.  Hence the number of possible $q$-sets of cycles with root-neighbour set contained in $A$ is at most $(Lh^2)^q$.

It remains to lower-bound the number of distinct graph edges in the union of the cycles.  The $h$ root edges $va$, $a\in A$, are all present.  Let $I=e(J(\mathcal Q))$ be the number of non-root support edges.  Since every $a\in A$ is used by at least one cycle, $I\ge h$.  Also the $q$ cycles are length-two paths in the support graph $J(\mathcal Q)$ with middle vertex in $U(\mathcal Q)$.  Thus
\[
        q
        \le \sum_{u\in U(\mathcal Q)}\binom{d_J(u)}2
        \le \frac{h-1}{2}\sum_{u\in U(\mathcal Q)}d_J(u)
        =\frac{h-1}{2}I.
\]
Therefore $I\ge 2q/(h-1)$, and so
\[
        |E(\bigcup_{Q\in\mathcal Q}Q)|
        \ge h+\max\left\{h,\frac{2q}{h-1}\right\}=h+M(h,q).
\]
Multiplying the count of possible $\mathcal Q$ by the survival probability of their edge union gives (\ref{lem:root-neighbour-moment}.1).

\end{proof}

\begin{lemma}[Rooted codegree-profile moment bound]\label{lem:codegree-profile-moment}
Let $G$ be triangle-free, let $v\in V(G)$, and for $A\subseteq N(v)$ put
\[
        W_v(A)=\sum_{\{a,b\}\subseteq A}|(N(a)\cap N(b))\setminus\{v\}|.
\]
If $Z_v$ is the number of $4$-cycles through $v$ in $G_p$, then for every integer $q\ge1$,
\[
        \mathbb E\binom{Z_v}{q}
        \le
        \sum_{h=2}^{2q}
        p^{h+M(h,q)}
        \sum_{\substack{A\subseteq N(v)\\ |A|=h}}
        W_v(A)^q,                                      \tag{\ref{lem:codegree-profile-moment}.1}
\]
where $M(h,q)=\max\{h,2q/(h-1)\}$.
Consequently, let $(G_n)$ be a sequence of triangle-free graphs with $t_n=\chif(G_n)\to\infty$.  Suppose that there are retention probabilities $p_n\in(0,1)$ and integers $q_n\ge1$ such that
\[
        \sum_{v\in V(G_n)}
        \sum_{h=2}^{2q_n}
        p_n^{h+M(h,q_n)}
        \sum_{\substack{A\subseteq N(v)\\ |A|=h}}
        W_v(A)^{q_n}=o(1/t_n),                         \tag{CP}
\]
and
\[
        q_n=o\!\left(\frac{p_nt_n}{\log t_n}\right).                    \tag{scale}
\]
Then $h^{\mathrm f}_5(G_n)\to\infty$.  In particular, when $q_n$ is fixed and $p_nt_n/\log t_n\to\infty$, the scale condition is automatic; condition \(CP\) must still be verified.
\end{lemma}

\begin{proof}
Fix a $q$-set $\mathcal Q$ of rooted $4$-cycles through $v$, and let $A=A(\mathcal Q)$ be its set of root-neighbours.  If $|A|=h$, then every possible rooted $4$-cycle with both root-neighbours in $A$ is determined by a pair $\{a,b\}\subseteq A$ and a vertex of $(N(a)\cap N(b))\setminus\{v\}$.  Thus the number of possible $q$-sets of rooted cycles supported on a fixed $A$ is at most $W_v(A)^q$.  The support-edge lower bound from Lemma~\ref{lem:root-neighbour-moment} gives
\[
        |E(\cup_{Q\in\mathcal Q}Q)|\ge h+M(h,q).
\]
Multiplying by the survival probability of this edge union and summing over $h$ and $A$ proves (\ref{lem:codegree-profile-moment}.1).

For each $n$, condition (CP) gives $\sum_v\mathbb E\binom{Z_v}{q_n}=o(1/t_n)$.  Lemma~\ref{lem:factorial-spread} gives
\[
        \lambda_{<5}((G_n)_{p_n})\le q_n-1
\]
with probability $1-o(1/t_n)$.  Mohar--Wu preservation, Theorem~\ref{thm:mohar-wu-random}, gives
\[
        \chif((G_n)_{p_n})\gg \frac{p_nt_n}{\log t_n}
\]
with probability at least $1-1/(2t_n)$.  On a common outcome, Proposition~\ref{prop:fractional-spread-criterion} deletes a short-cycle transversal with fractional cost at most $q_n$, and therefore
\[
        h^{\mathrm f}_5(G_n)
        \gg \frac{p_nt_n}{q_n\log t_n}.
\]
This tends to infinity by (scale).  If $q_n$ is fixed, the scale condition reduces to $p_nt_n/\log t_n\to\infty$; the profile condition \(CP\) remains an independent hypothesis.
\end{proof}

\begin{remark}[Use and limitation of the codegree-profile bound]\label{rem:codegree-profile-use}
Lemma~\ref{lem:codegree-profile-moment} is a profile criterion, not merely a reformulation of the maximum-codegree argument.  It can replace a maximum-codegree hypothesis whenever one has genuine control of the sums in \textup{(CP)}.  A bound on the total number of cherries alone does not imply \textup{(CP)}: in a $K_{2,M}$, for a root $z$ on the $M$-side and the two opposite vertices $a,b$, one has $W_z(\{a,b\})=M-1$, so the $q$th moment is dominated by one highly concentrated pair.  This is why the pseudorandom applications still require genuine control of rooted codegree profiles, or at least a uniform maximum-codegree input.  The profile lemma is retained as the appropriate future tool for settings where high-codegree pairs are few or sufficiently dispersed.
\end{remark}

\begin{lemma}[The logarithmic codegree exponent]\label{lem:log-codeg-exponent}
Fix $0<\delta<1/2$.  For every $R>0$ there is an integer $q=q(\delta,R)$ such that, for every integer $h$ with $2\le h\le 2q$,
\[
        (1-\delta)\max\left\{h,\frac{2q}{h-1}\right\}-\delta h>R.
\]
\end{lemma}

\begin{proof}
Write
\[
        F_q(h)=(1-\delta)\max\left\{h,\frac{2q}{h-1}\right\}-\delta h.
\]
If $h\ge\sqrt q$, then
\[
        F_q(h)\ge (1-2\delta)h\ge (1-2\delta)\sqrt q.
\]
If $2\le h<\sqrt q$, then
\[
        F_q(h)\ge (1-\delta)\frac{2q}{h-1}-\delta h
        \ge 2(1-\delta)\sqrt q-\delta\sqrt q.
\]
Both lower bounds tend to infinity with $q$ because $\delta<1/2$.  Choose $q$ large enough.
\end{proof}

\begin{theorem}[Polylogarithmic-codegree pseudorandom cheap killing]\label{thm:log-codegree-pseudorandom}
Fix constants $C_0,C_1,C_2>0$, $a,b,\rho\ge0$, and $0<\delta<1/2$.  There are constants $D_0,c>0$ and $t_0$ such that the following holds.  Let $G$ be triangle-free with
\[
        t:=\chif(G)\ge t_0,
        \qquad
        |V(G)|\le C_0t^2(\log t)^a,
        \qquad
        \Delta(G)\le C_1t(\log t)^b,
\]
and
\[
        \operatorname{codeg}(G)\le C_2(\log t)^\rho.
\]
Then
\[
        h_5^{\mathrm f}(G)
        \ge
        c\frac{t^\delta}{\log t}.
\]
More precisely, with $p=t^{-1+\delta}$,
\[
        \Pr\bigl(\lambda_{<5}(G_p)\le D_0\bigr)>1-\frac1{2t}
\]
for all sufficiently large $t$.
\end{theorem}

\begin{proof}
Choose $q\ge2$ so large that Lemma \ref{lem:log-codeg-exponent} holds with
\[
        R=5+a.
\]
Set $D_0=q-1$ and $p=t^{-1+\delta}$.  For a fixed vertex $v$, Lemma \ref{lem:root-neighbour-moment} gives
\[
        \mathbb E\binom{Z_v}{q}
        \le
        \sum_{h=2}^{2q}
        \Delta^h(Lh^2)^q p^{h+M(h,q)},                       \tag{1}
\]
where $L=\operatorname{codeg}(G)$ and $M(h,q)=\max\{h,2q/(h-1)\}$.  Using
\[
        \Delta\le C_1t(\log t)^b,
        \qquad
        L\le C_2(\log t)^\rho,
        \qquad
        p=t^{-1+\delta},
\]
the $h$th term in (1) is at most
\[
        C_{h,q}(\log t)^{bh+\rho q}
        t^{\delta h-(1-\delta)M(h,q)}.
\]
By the choice of $q$,
\[
        (1-\delta)M(h,q)-\delta h>5+a
        \tag{2}
\]
for every $2\le h\le2q$.  Since $h\le2q$, the logarithmic factor in the displayed bound is $t^{o(1)}$.  Hence, uniformly in $v$,
\[
        \mathbb E\binom{Z_v}{q}
        \le t^{-5-a+o(1)}\le t^{-4-a}
\]
for all sufficiently large $t$.
As $|V(G)|\le C_0t^2(\log t)^a$, we obtain
\[
        \sum_{v\in V(G)}\mathbb E\binom{Z_v}{q}
        =o(1/t).
\]
Lemma \ref{lem:factorial-spread} therefore gives $\lambda_{<5}(G_p)\le q-1=D_0$ with probability $1-o(1/t)$.

On the other hand, the Mohar--Wu preservation theorem, Theorem \ref{thm:mohar-wu-random}, gives with probability at least $1-1/(2t)$ that
\[
        \chif(G_p)\ge c_1\frac{pt}{\log t}
        =c_1\frac{t^\delta}{\log t}.
\]
For a common outcome, Proposition \ref{prop:fractional-spread-criterion} deletes a bounded-degree short-cycle transversal with multiplicative fractional cost at most $D_0+1$.  This proves the theorem.
\end{proof}

\begin{corollary}[Kim-scale polylogarithmic-codegree form]\label{cor:kim-log-codegree}
Let $(G_n)$ be triangle-free graphs on $N_n$ vertices with $t_n=\chif(G_n)\to\infty$.  Suppose that, for some constants $C_0,C_1,C_2>0$ and $\rho\ge0$, one has
\[
        \alpha(G_n)\le C_0\sqrt{N_n\log N_n},
        \qquad
        \Delta(G_n)\le C_1\sqrt{N_n\log N_n},
        \qquad
        \operatorname{codeg}(G_n)\le C_2(\log N_n)^\rho.
\]
Then, for every $0<\delta<1/2$,
\[
        h_5^{\mathrm f}(G_n)
        \ge c_\delta\frac{t_n^\delta}{\log t_n}
        \to\infty .
\]
\end{corollary}

\begin{proof}
Since $t_n=\chif(G_n)\ge N_n/\alpha(G_n)$,
\[
        t_n\ge C_0^{-1}\sqrt{N_n/\log N_n}.
\]
It follows first that
\[
        N_n\le C t_n^2\log N_n .
\]
Taking logarithms gives $\log N_n\le O(1)+2\log t_n+\log\log N_n$, and since $t_n\to\infty$ this implies $\log N_n=O(\log t_n)$.  Substituting this back yields
\[
        N_n\le C't_n^2\log t_n,
        \qquad
        \sqrt{N_n\log N_n}\le C''t_n(\log t_n)
\]
for constants $C',C''$.  The codegree assumption gives $\operatorname{codeg}(G_n)\le C'''(\log t_n)^{\rho}$ after increasing $C'''$.  Theorem \ref{thm:log-codegree-pseudorandom} applies with fixed logarithmic exponents.
\end{proof}

\begin{remark}[FGM stopping-time graphs]
The deterministic Corollary~\ref{cor:kim-log-codegree} applies to any triangle-free-process stopping-time graph for which one has the three inputs
\[
        \alpha(G)\le C\sqrt{n\log n},\qquad
        \Delta(G)\le C\sqrt{n\log n},\qquad
        \operatorname{codeg}(G)\le (\log n)^{O(1)}.
\]
Fiz Pontiveros--Griffiths--Morris prove Ramsey-scale degree and independence estimates for their stopping-time graph.  We do not state a separate FGM theorem here because the all-pairs maximum-codegree estimate must be quoted or proved in exactly that notation before applying Corollary~\ref{cor:kim-log-codegree}.
\end{remark}

\begin{proposition}[Bohman--Keevash tracking-time inputs]\label{prop:bk-inputs}
Let $G(i)$ be the triangle-free process on $n$ vertices, in the notation of Bohman--Keevash.  Let
\[
        L=\sqrt{\log n},\qquad
        t_{\max}=\frac12\sqrt{(1/2-\varepsilon)\log n},
        \qquad
        i_{\max}=t_{\max}n^{3/2}.
\]
Then, with high probability,
\[
        I>i_{\max},
\]
and, for every pair $u\ne v$,
\[
        Z_{uv}(i_{\max})
        :=|N_{G(i_{\max})}(u)\cap N_{G(i_{\max})}(v)|
        \le L^4=(\log n)^2.
\]
Moreover there is an absolute constant $C$ such that, with high probability,
\[
        \Delta(G(i_{\max}))\le C\sqrt{n\log n},
        \qquad
        \alpha(G(i_{\max}))\le C\sqrt{n\log n}.
\]
\end{proposition}

\begin{proof}
This is a quoted consequence of Bohman--Keevash, translated into the notation used here.  Their Theorem~2.13 gives $I>i_{\max}$ with high probability and supplies the tracked degree estimate at time $t_{\max}$.  In particular all degrees at time $i_{\max}$ are $O(\sqrt{n\log n})$ on the tracked event.  Their Lemma~3.10 gives, on the same high-probability tracked event, the estimate $Z_{uv}(i')\le L^4$ for every relevant pair $uv$ whenever $i'-1<I$.  Taking $i'=i_{\max}$ is valid on the event $I>i_{\max}$, and gives the displayed codegree bound.  If one uses a formulation of the tracked estimate in which the variable is stated only for open or non-adjacent pairs, then adjacent pairs are harmless: every process graph is triangle-free, so an adjacent pair has no common neighbour.  Finally, the independence estimate is the bound proved in their Section~7.2 at the same tracking time: the event that $I>i_{\max}$ and $G(i_{\max})$ has an independent set larger than $C\sqrt{n\log n}$ has probability $o(1)$ for a suitable absolute constant $C$.  See \cite[Theorem~2.13, Lemma~3.10, and Section~7.2]{BohmanKeevash2013}.
\end{proof}

\begin{corollary}[Triangle-free process at the Bohman--Keevash tracking time]\label{cor:bk-imax-process}
Let $G(i_{\max})$ be the triangle-free-process graph at the Bohman--Keevash tracking time.  Then, with high probability,
\[
        h^{\mathrm f}_5(G(i_{\max}))\to\infty .
\]
More quantitatively, for every fixed $0<\delta<1/2$,
\[
        h^{\mathrm f}_5(G(i_{\max}))
        \ge n^{\delta/2-o(1)}.
\]
\end{corollary}

\begin{proof}
The graph $G(i_{\max})$ is triangle-free.  Proposition~\ref{prop:bk-inputs} gives the degree and independence estimates at the tracking time.  Hence
\[
        t:=\chif(G(i_{\max}))\ge n^{1/2-o(1)},
        \qquad
        |V(G(i_{\max}))|\le t^2(\log t)^{O(1)},
        \qquad
        \Delta(G(i_{\max}))\le t(\log t)^{O(1)}.
\]
Proposition~\ref{prop:bk-inputs} supplies the polylogarithmic maximum-codegree bound.  Theorem~\ref{thm:log-codegree-pseudorandom} applies and gives the result.
\end{proof}

\begin{remark}[Terminal triangle-free process]
The terminal graph $G_{\mathrm{fin}}$ of the triangle-free process has Ramsey-scale degree and independence bounds by Bohman--Keevash.  The deterministic Corollary~\ref{cor:kim-log-codegree} would apply to $G_{\mathrm{fin}}$ if one had the final-time bound
\begin{equation}\tag{TC}
        \max_{x\ne y}\left|N_{G_{\mathrm{fin}}}(x)\cap N_{G_{\mathrm{fin}}}(y)\right|\le (\log n)^{O(1)}.
\end{equation}
We do not state a terminal-process theorem here, because Proposition~\ref{prop:bk-inputs} proves the needed maximum-codegree bound at the Bohman--Keevash tracking time $i_{\max}$, whereas (TC) is a separate final-time assertion after the post-tracking edges have been added.
\end{remark}

\begin{remark}[On process codegree inputs]
The deterministic content behind the process applications is Theorem~\ref{thm:log-codegree-pseudorandom}.  Corollary~\ref{cor:bk-imax-process} is unconditional because Proposition~\ref{prop:bk-inputs} supplies the required all-pairs maximum-codegree bound at the Bohman--Keevash tracking time.  We do not promote other process times, such as the FGM stopping-time graph or the terminal graph, to theorem statements here, because they require separate all-pairs maximum-codegree estimates in their own notation.
\end{remark}

\begin{proposition}[Ramsey-scale triangle-free graphs need logarithmic codegree]\label{prop:ramsey-codegree-obstruction}
Let $C>0$.  There is a constant $c=c(C)>0$ such that every sufficiently large triangle-free graph $G$ on $N$ vertices with
\[
        \alpha(G)\le C\sqrt{N\log N}
\]
satisfies
\[
        \operatorname{codeg}(G)\ge c\log N.
\]
Consequently, if $t=\chif(G)$, then
\[
        \operatorname{codeg}(G)\ge c'\log t
\]
for another constant $c'=c'(C)>0$ and all sufficiently large $N$.
\end{proposition}

\begin{proof}
Let $\bar d=2e(G)/N$ be the average degree.  If $\bar d$ is bounded, then the Caro--Wei bound gives $\alpha(G)\ge N/(\bar d+1)=\Omega(N)$, contradicting $\alpha(G)\le C\sqrt{N\log N}$ for all large $N$.  Hence we may assume $\bar d\to\infty$.  We use Shearer's theorem \cite{Shearer1983} in the following standard asymptotic form: every triangle-free graph with average degree $\bar d\to\infty$ has
\[
        \alpha(G)\ge c_0\frac{N\log \bar d}{\bar d}
\]
for an absolute constant $c_0>0$ and all sufficiently large $\bar d$.  If $\bar d\le c_1\sqrt{N\log N}$ with $c_1=c_1(C)>0$ sufficiently small, then
\[
        \alpha(G)
        \ge
        c_0\frac{N\log(c_1\sqrt{N\log N})}{c_1\sqrt{N\log N}}
        \ge
        2C\sqrt{N\log N}
\]
for all large $N$, a contradiction.  Thus $\bar d\ge c_1\sqrt{N\log N}$.

Counting cherries gives
\[
        \sum_{\{x,y\}\subseteq V(G)}|N(x)\cap N(y)|
        =\sum_{v\in V(G)}\binom{d(v)}2
        \ge N\binom{\bar d}{2}.
\]
Dividing by $\binom N2$ shows that the average codegree is at least
\[
        (1-o(1))\frac{\bar d^2}{N}\ge c_2\log N,
\]
so the maximum codegree is at least $c_2\log N$.  Since $t=\chif(G)\le N$, we have $\log t\le \log N$, and therefore the maximum codegree is at least $c_2\log t$ after decreasing $c_2$ if necessary.
\end{proof}

\begin{remark}[The logarithmic-to-polylogarithmic codegree window]
Proposition \ref{prop:ramsey-codegree-obstruction} shows that logarithmic maximum codegree is necessary at the Ramsey scale, while Corollary \ref{cor:kim-log-codegree} shows that any fixed polylogarithmic maximum codegree is sufficient for fractional cheap killing, together with the usual Ramsey-scale order and degree bounds.  Thus the previous $\log\log t$ gap is closed up to fixed polylogarithmic powers at the level of deterministic hypotheses.  Any triangle-free graph satisfying the explicit order, degree and polylogarithmic maximum-codegree hypotheses is covered by Corollary \ref{cor:kim-log-codegree}; in particular, the Bohman--Keevash tracking time is covered unconditionally by Corollary~\ref{cor:bk-imax-process}.
\end{remark}

\begin{corollary}[A spectral polylogarithmic-codegree form]
Let $(G_n)$ be triangle-free $d_n$-regular graphs with $t_n=\chif(G_n)\to\infty$.  Suppose that for some constants $C_0,C_1,C_2,a,b,\rho$,
\[
        |V(G_n)|\le C_0t_n^2(\log t_n)^a,
        \qquad
        d_n\le C_1t_n(\log t_n)^b,
        \qquad
        \operatorname{codeg}(G_n)\le C_2(\log t_n)^\rho.
\]
Then $h_5^{\mathrm f}(G_n)\to\infty$.  In spectral applications, the lower bound $t_n\ge 1+d_n/\lambda_n$ follows from Hoffman whenever the least eigenvalue is at least $-\lambda_n$.
\end{corollary}

\begin{proof}
This is Theorem \ref{thm:log-codegree-pseudorandom}.  The final sentence is the standard Hoffman fractional-colouring bound: if the least eigenvalue is at least $-\lambda_n$, then
\[
        \alpha(G_n)\le \frac{|V(G_n)|\lambda_n}{d_n+\lambda_n},
        \qquad
        \chif(G_n)\ge \frac{|V(G_n)|}{\alpha(G_n)}\ge 1+\frac{d_n}{\lambda_n}.
\]
\end{proof}

\begin{proposition}[The Burling factorial-moment test in explicit form]\label{prop:burling-factorial-test}
Let $B_j$ be a standard triangle-free Burling graph, put $t_j=\chif(B_j)$, let $p_j\in(0,1)$, and for $v\in V(B_j)$ let $\mathcal C_v(B_j)$ be the family of $4$-cycles of $B_j$ containing $v$.  Suppose $t_j>1$ eventually.  If
\[
        \frac{p_jt_j}{\log t_j}\to\infty
\]
and there are integers $D_j\ge0$ with $D_j+1=o(p_jt_j/\log t_j)$ such that
\[
        \sum_{v\in V(B_j)}
        \sum_{\substack{\mathcal Q\subseteq \mathcal C_v(B_j)\\ |\mathcal Q|=D_j+1}}
        p_j^{|E(\cup_{Q\in\mathcal Q}Q)|}=o(1/t_j),        \tag{B$'$}
\]
then
\[
        h^{\mathrm f}_5(B_j)\to\infty .
\]
Moreover
\[
        \E Z_v=p_j^4|\mathcal C_v(B_j)|,
\]
where $Z_v$ is the number of retained $4$-cycles through $v$ in $(B_j)_{p_j}$.
\end{proposition}

\begin{proof}
The first displayed identity is just the first moment of the rooted $4$-cycle count after independent edge-retention.  For the factorial moment, expand
\[
        \E\binom{Z_v}{D_j+1}
        =
        \sum_{\substack{\mathcal Q\subseteq \mathcal C_v(B_j)\\ |\mathcal Q|=D_j+1}}
        \Pr\bigl(E(Q)\subseteq E((B_j)_{p_j})\text{ for every }Q\in\mathcal Q\bigr).
\]
Since edges are retained independently, the probability in the summand is
\[
        p_j^{|E(\cup_{Q\in\mathcal Q}Q)|}.
\]
Thus (B$'$) is exactly the condition
\[
        \sum_v \E\binom{Z_v}{D_j+1}=o(1/t_j)
\]
from Corollary~\ref{cor:ramsey-spread}.  The additional assumption $p_jt_j/\log t_j\to\infty$ is the Mohar--Wu fractional-preservation scale.  Since $D_j+1=o(p_jt_j/\log t_j)$, Corollary~\ref{cor:ramsey-spread} yields $h^{\mathrm f}_5(B_j)\to\infty$.
\end{proof}

We shall use the following standard recursive model of the triangle-free Burling graphs, in the form recorded by Felsner--Joret--Micek--Trotter--Wiechert.  The graph $G_1$ consists of one vertex, and $\mathcal S(G_1)$ consists of its one singleton stable set.  Having constructed $G_i$ together with a family $\mathcal S(G_i)$ of stable sets, construct $G_{i+1}$ as follows.  Start with one master copy $H$ of $G_i$.  For every $S\in\mathcal S(H)$ add a copy $H_S$ of $G_i$.  For every $X\in\mathcal S(H_S)$ add a new vertex $v_{S,X}$ adjacent exactly to the vertices of $X$.  The new stable-set family $\mathcal S(G_{i+1})$ contains the sets $S\cup X$ and $S\cup\{v_{S,X}\}$ for all such pairs $(S,X)$.  Since each $X$ is stable, each new vertex is adjacent to a stable set, and the copies are otherwise disjoint; induction therefore shows that every $G_i$ is triangle-free.  This is the model used in Proposition~\ref{prop:burling-growth} and below.

\begin{proposition}[Tight rooted clusters in the standard Burling recursion]\label{prop:burling-tight-clusters}
Let $B_j=G_j$ be the standard Burling sequence with the stable-set collections $\mathcal S(G_j)$ from Proposition~\ref{prop:burling-growth}.  For $j\ge4$, the graph $B_j$ contains a copy of
\[
        K_{2,M_j},
        \qquad
        M_j=2^{2^{j-2}-2}.
\]
Consequently there is a vertex $z\in V(B_j)$ with at least $M_j-1$ rooted $4$-cycles, and these cycles contain a subfamily with the following extremal overlap property: for every $1\le s\le M_j-1$, some $s$ rooted $4$-cycles through $z$ have union with exactly $2s+2$ edges.
\end{proposition}

\begin{proof}
For two vertices $a,b\in V(G_i)$, write
\[
        \mu_i(a,b)=|\{S\in\mathcal S(G_i):\{a,b\}\subseteq S\}|.
\]
We first prove that for every $i\ge2$ there are distinct nonadjacent vertices $a_i,b_i\in V(G_i)$ such that
\[
        \mu_i(a_i,b_i)\ge m_i:=2^{2^{i-1}-2}.       \label{eq:burling-pair-multiplicity}
\]
For $i=2$, the graph $G_2$ is formed from a master vertex, a copy vertex, and one new vertex.  One of the two stable sets in $\mathcal S(G_2)$ contains the master vertex and the copy vertex, so $\mu_2\ge1=m_2$.

Assume that $a_i,b_i$ lie in a master copy $H$ of $G_i$ and are contained in at least $m_i$ members of $\mathcal S(H)$.  In the construction of $G_{i+1}$, for every stable set $S\in\mathcal S(H)$ containing $a_i,b_i$ and every $X\in\mathcal S(H_S)$, both
\[
        S\cup X
        \quad\text{and}\quad
        S\cup\{v_{S,X}\}
\]
are members of $\mathcal S(G_{i+1})$ and contain $a_i,b_i$.  Therefore the same pair in the master copy of $G_{i+1}$ is contained in at least
\[
        2|\mathcal S(G_i)|m_i
        =2\cdot 2^{2^{i-1}-1}\cdot 2^{2^{i-1}-2}
        =2^{2^i-2}=m_{i+1}
\]
stable sets.  This proves \eqref{eq:burling-pair-multiplicity}.

Now pass from $G_{j-1}$ to $G_j$.  Choose a copy $H_S$ of $G_{j-1}$ and a pair $a,b\in V(H_S)$ contained in $m_{j-1}=M_j$ stable sets of $\mathcal S(H_S)$, say $X_1,\ldots,X_{M_j}$.  The construction adds vertices $v_{S,X_1},\ldots,v_{S,X_{M_j}}$, each adjacent to both $a$ and $b$.  These vertices are pairwise nonadjacent, and $a,b$ are nonadjacent because every $X_i$ is stable.  Hence they span a $K_{2,M_j}$.

Fix one vertex $z=v_{S,X_1}$ on the $M_j$-side.  For each $i=2,\ldots,M_j$ the four vertices
\[
        z,a,v_{S,X_i},b
\]
form a $4$-cycle.  If we choose $s$ of these cycles, their union consists of the two edges $za,zb$ and, for each chosen auxiliary vertex $v_{S,X_i}$, the two edges $av_{S,X_i},bv_{S,X_i}$.  Thus the union has exactly $2+2s$ edges.
\end{proof}

\begin{corollary}[The Burling factorial-moment test fails on the tight clusters]\label{cor:burling-test-fails}
Let $B_j$ be the standard Burling sequence and put $t_j=\chif(B_j)$.  Suppose $t_j>1$ eventually, and suppose $p_j\in(0,1)$ and $D_j\ge0$ satisfy
\[
        \frac{p_jt_j}{\log t_j}\to\infty,
        \qquad
        D_j+1=o\!\left(\frac{p_jt_j}{\log t_j}\right).
\]
Then the cluster sum in (B$'$) is not small.  In fact,
\[
        \sum_{v\in V(B_j)}
        \sum_{\substack{\mathcal Q\subseteq \mathcal C_v(B_j)\\ |\mathcal Q|=D_j+1}}
        p_j^{|E(\cup_{Q\in\mathcal Q}Q)|}\to\infty.
\]
Consequently the sufficient condition (B$'$) from Proposition~\ref{prop:burling-factorial-test} cannot prove fractional cheap cycle killing for the standard Burling sequence.
\end{corollary}

\begin{proof}
Let
\[
        M_j=2^{2^{j-2}-2},
        \qquad
        s_j=D_j+1.
\]
By Proposition~\ref{prop:burling-tight-clusters}, $B_j$ contains a $K_{2,M_j}$ with bipartition $\{a,b\}\cup W$.  For every root $z\in W$ and every $s_j$-set $U\subseteq W\setminus\{z\}$, the family of $4$-cycles
\[
        \{zau bz: u\in U\}
\]
has union with $2s_j+2$ edges.  Therefore the left hand side of (B$'$) is at least
\[
        M_j\binom{M_j-1}{s_j}p_j^{2s_j+2}.        \label{eq:burling-cluster-lower}
\]
The standard Burling graph satisfies $\chi(B_j)=j$: the inductive proof gives $\chi(B_j)\ge j$, while a $j$-colouring is obtained by colouring every copy of $G_{j-1}$ with colours $1,\ldots,j-1$ and giving all vertices added at the last step colour $j$.  Thus $t_j\le j$.

Put $A_j=p_jt_j/\log t_j$.  By assumption $A_j\to\infty$ and $s_j=o(A_j)+1$, so $s_j\le A_j$ for all large $j$.  Also $p_j\ge A_j\log t_j/t_j$.  Hence
\[
        \frac{M_jp_j^2}{s_j}
        \ge
        \frac{M_j A_j^2(\log t_j)^2}{t_j^2s_j}
        \ge
        \frac{M_j A_j(\log t_j)^2}{t_j^2}\to\infty,
\]
because $t_j\le j$ while $M_j$ is doubly exponential in $j$.
For large $j$ we have $s_j\le M_j/2$, and hence
\[
        \binom{M_j-1}{s_j}\ge \left(\frac{M_j}{2s_j}\right)^{s_j}.
\]
Substituting this in \eqref{eq:burling-cluster-lower} gives
\[
        M_j\binom{M_j-1}{s_j}p_j^{2s_j+2}
        \ge
        M_jp_j^2\left(\frac{M_jp_j^2}{2s_j}\right)^{s_j}\to\infty.
\]
This proves the claim.
\end{proof}

The preceding corollary has preservation-scale hypotheses displayed in its statement.  For the unweighted standard sequence it is therefore vacuous if $\chif(B_j)$ does not tend to infinity.  The point of the result is to identify the rooted cluster obstruction in the standard recursion whenever the fractional parameter is large enough for Mohar--Wu preservation to be relevant.  Walczak's weighted Burling construction gives such fractional relevance for Burling-type objects; the next elementary lemma records the standard conversion from weighted independence certificates to unweighted blow-ups.

\begin{lemma}[Weighted certificates and unweighted blow-ups]\label{lem:weighted-blowup}
Let $X$ be a finite graph and let $w:V(X)\to\mathbb Q_{>0}$.  Put $W=w(V(X))$.  Suppose that every independent set $I$ of $X$ satisfies
\[
        w(I)\le W/T.
\]
Let $N$ clear all denominators and let $X^{(Nw)}$ be the blow-up obtained by replacing each vertex $v$ by an independent set of size $Nw(v)$ and replacing every edge $uv\in E(X)$ by the complete bipartite graph between the two corresponding parts.  Then
\[
        \chif(X^{(Nw)})\ge T.
\]
The same conclusion holds with an arbitrary positive real weight function after rational approximation, with $T$ replaced by $T-o(1)$.
\end{lemma}

\begin{proof}
The blow-up has $NW$ vertices.  An independent set of the blow-up projects to an independent set $I$ of $X$, and therefore has size at most $Nw(I)\le NW/T$.  Since every graph $Y$ satisfies $\chif(Y)\ge |V(Y)|/\alpha(Y)$, we get
\[
        \chif(X^{(Nw)})\ge \frac{NW}{NW/T}=T.
\]
For real weights, approximate them by positive rational weights so that the ratios $w(I)/w(V(X))$ change by an arbitrarily small amount for all finitely many independent sets.
\end{proof}

\begin{remark}[Weighted Burling relevance]
Walczak constructs weighted Burling graphs for which every stable set has small weight compared with the total weight \cite{Walczak2015}.  Lemma~\ref{lem:weighted-blowup} turns such weighted certificates into unweighted blow-ups with large fractional chromatic number.  We do not identify these blow-ups with the standard unweighted Burling sequence $B_j$ used in Proposition~\ref{prop:burling-tight-clusters}; assertions about the standard sequence are stated separately and remain conditional on the displayed parameter $t_j=\chif(B_j)$ tending to infinity.
\end{remark}

\begin{lemma}[Acyclic one-out deletion]\label{lem:acyclic-one-out}
Let $X$ be a graph with an acyclic orientation $\vec X$, and let $D$ be the maximum outdegree of $\vec X$.  There is a set of edges $F\subseteq E(X)$ such that
\[
        \chi_f(V(X),F)\le D+1
        \qquad\text{and}\qquad
        X-F\text{ is a forest}.
\]
Consequently $\rho_{<r}(X)\le D+1$ for every $r\ge4$.
\end{lemma}

\begin{proof}
Fix a topological ordering of the acyclic orientation, with all edges oriented from earlier to later.  For every vertex $v$, keep at most one outgoing edge of $v$, and put every other outgoing edge of $v$ into $F$.  Then $F$, oriented as in $\vec X$, has maximum outdegree at most $D$.  In every non-empty subgraph of $(V(X),F)$, the earliest vertex in the topological order has undirected degree equal to its outdegree in that subgraph, and hence at most $D$.  Thus $(V(X),F)$ is $D$-degenerate and $\chi_f(V(X),F)\le D+1$.

It remains to show that $X-F$ is acyclic.  If $C$ were a cycle in $X-F$, let $v$ be the earliest vertex of $C$ in the topological ordering.  Both edges of $C$ incident with $v$ are outgoing from $v$.  But $X-F$ keeps at most one outgoing edge at $v$, a contradiction.  Hence $X-F$ is a forest.
\end{proof}

\begin{corollary}[The $K_{2,M}$ cluster is cheap for $\rho_{<5}$]\label{cor:k2m-star-cheap}
For every $M\ge1$,
\[
        \rho_{<5}(K_{2,M})\le2.
\]
\end{corollary}

\begin{proof}
Orient every edge from the $M$-side to the $2$-side.  The maximum outdegree is $2$, so Lemma \ref{lem:acyclic-one-out} gives $\rho_{<5}(K_{2,M})\le3$.  The sharper bound $2$ is obtained directly: if the $2$-side is $\{a,b\}$ and the $M$-side is $W$, choose one $w_0\in W$ and delete
\[
        F=\{bw:w\in W\setminus\{w_0\}\}.
\]
Then $(V,F)$ is a star, so $\chi_f(V,F)\le2$, and $K_{2,M}-F$ has no $4$-cycle because at most one vertex of $W$ remains adjacent to both $a$ and $b$.
\end{proof}

\begin{remark}[Burling graphs]
Burling graphs are a canonical hard family for the integral problem: Pettie--Tardos--Walczak use the Burling framework to obtain tower-type lower bounds at $r=5$ \cite{PTW2026}.  They are also relevant fractionally in a weighted sense: Walczak's weighted version of the Burling construction gives weighted Burling graphs with arbitrarily large fractional obstruction certificates \cite{Walczak2015}, and Lemma~\ref{lem:weighted-blowup} converts such certificates into unweighted blow-ups with large fractional chromatic number.  Thus Proposition~\ref{prop:burling-factorial-test} is best read as the explicit test for the standard sequence, while the weighted blow-up lemma explains why Burling-type constructions remain fractionally relevant.

Proposition~\ref{prop:burling-factorial-test} gave a possible sufficient condition for the rooted factorial-moment method to prove $h^{\mathrm f}_5(B_j)\to\infty$.  Corollary~\ref{cor:burling-test-fails} shows that this particular sufficient condition fails as strongly as possible on the standard Burling recursion: the graphs contain $K_{2,M}$ subgraphs with $M$ doubly exponential in the chromatic parameter, and the rooted $4$-cycles in those subgraphs share two root edges.  This failure should not be mistaken for a proof that Burling graphs are hard for fractional cheap killing.  Corollary~\ref{cor:k2m-star-cheap} shows that an isolated $K_{2,M}$ cluster has fractional short-cycle deletion cost at most $2$; the factorial-moment method is lossy precisely because it measures such a cluster by its rooted multiplicity rather than by its star-deletability.  Thus the remaining Burling problem is to globalize this star-deletion economy through the whole recursion.  A positive solution would prove $h^{\mathrm f}_5(B_j)\to\infty$ for Burling graphs; a negative solution would have to show that the many overlapping $K_{2,M}$ clusters cannot be hit by any edge set of low fractional chromatic number.
\end{remark}

\subsection{The fractional Erd\H{o}s--Hajnal conjecture}

Mohar and Wu proved a fractional analogue of R\"odl's triangle-free subgraph theorem, namely the case $r=4$ of the fractional Erd\H{o}s--Hajnal problem \cite{MoharWuTriangleFreeFractional}.  They also verified higher-girth statements for Kneser graphs \cite{MoharWuKneser}.  The following general form remains the cleanest fractional target.

\begin{question}[Fractional Erd\H{o}s--Hajnal, higher girth]\label{q:fractional-EH}
For every $r\ge4$ and every $q>1$, is there $F_{\mathrm f}(r,q)$ such that every graph $G$ with
\[
        \chif(G)\ge F_{\mathrm f}(r,q)
\]
contains a subgraph $H$ with
\[
        \girth(H)\ge r,
        \qquad
        \chif(H)>q
\]
?
\end{question}

Theorem \ref{thm:fractional-surgery} gives a concrete route toward Question \ref{q:fractional-EH}.  Mohar--Wu already provide the small-retention preservation term of order $p t/\log t$.  Thus a positive solution would follow from a density-free bound on the fractional short-cycle cost $\rho_{<r}(G_p)$ in the thinned graph.  One possible formulation is the following.

\begin{question}[Fractional thinning with cheap cycle killing]\label{q:cheap-cycle-killing}
Fix $r\ge5$.  Do there exist functions $p(t)\in(0,1)$ and $B(t)\ge0$ such that
\[
        \frac{p(t)t}{\log t}\to\infty,
        \qquad
        B(t)=o\left(\left(\frac{p(t)t}{\log t}\right)^2\right),
\]
and for every graph $G$ with $\chif(G)=t$, the random edge-subgraph $G_{p(t)}$ has, with probability bounded away from zero, an edge set $F$ meeting all cycles of length less than $r$ with
\[
        \chif(V(G),F)\le 1+\sqrt{2B(t)}
\]
?
\end{question}

A positive answer to Question \ref{q:cheap-cycle-killing}, combined with Theorem \ref{thm:mohar-wu-random} and Proposition \ref{prop:fractional-deletion}, would imply Question \ref{q:fractional-EH}.  This formulation emphasizes the point of Remark \ref{rem:fractional-gap}: in the fractional setting the logarithmic random-subgraph preservation is known; within this random-surgery framework, what is missing is a density-free way to pay for the deletion of short cycles.

We close this subsection with a precise obstruction theorem.  It does not prove Question \ref{q:fractional-EH}; rather, it says that any counterexample sequence must fail cheap cycle killing in a quantitatively robust way.  This is the form in which the remaining ``neither clique-organised nor codegree-controlled'' case should be attacked, with Burling graphs as the canonical test family.

\begin{theorem}[Density-free surgery dichotomy]\label{thm:fractional-obstruction-dichotomy}
Fix $r\ge5$ and let $(G_n)$ be graphs with $t_n=\chif(G_n)\to\infty$.  Put
\[
        A(t,p)=8\log_{1/(1-p)}(2et)+4 .
\]
This is the denominator in the Mohar--Wu preservation lower bound from Theorem~\ref{thm:mohar-wu-random}.
Let $p_n\in(0,1)$ and $R_n\ge1$ satisfy
\[
        R_n A(t_n,p_n)=o(t_n).
\]
Then either
\[
        \limsup_{n\to\infty} h_r^{\mathrm f}(G_n)=\infty,
\]
or, for all sufficiently large $n$,
\[
        \Pr\bigl(\rho_{<r}((G_n)_{p_n})\le R_n\bigr)
        \le \frac1{2t_n} .                         \tag{O}
\]
Equivalently, for each $n$ for which
\[
        \Pr\bigl(\rho_{<r}((G_n)_{p_n})\le R_n\bigr)>\frac1{2t_n},
\]
one has the explicit lower bound
\[
        h_r^{\mathrm f}(G_n)
        \ge \frac{t_n}{R_nA(t_n,p_n)}.
\]
Consequently, if a sequence $(G_n)$ has bounded $h_r^{\mathrm f}$, then every useful thinning, namely every choice with $R_nA(t_n,p_n)=o(t_n)$, has fractional short-cycle cost larger than $R_n$ with probability at least $1-1/(2t_n)$ for all sufficiently large $n$.
\end{theorem}

\begin{proof}
The per-$n$ assertion is exactly Theorem \ref{thm:fractional-surgery}.  Indeed, if the displayed probability is greater than $1/(2t_n)$, then Theorem \ref{thm:fractional-surgery} gives
\[
        h_r^{\mathrm f}(G_n)
        \ge \frac{t_n}{R_nA(t_n,p_n)}.
\]
Since $R_nA(t_n,p_n)=o(t_n)$, this lower bound tends to infinity along every subsequence on which the probability condition holds.  Hence if (O) fails infinitely often, then \(\limsup_{n\to\infty} h_r^{\mathrm f}(G_n)=\infty\).
\end{proof}

The same statement has a local rooted-cluster version obtained from Lemma \ref{lem:factorial-spread}.  This is the version most directly relevant to triangle-free Ramsey graphs and to Burling graphs, where the short cycles at girth target $5$ are $4$-cycles.

\begin{corollary}[Rooted cluster obstruction]\label{cor:rooted-cluster-obstruction}
Fix $r\ge5$ and let $(G_n)$ be graphs with $t_n=\chif(G_n)\to\infty$.  Let $p_n\in(0,1)$ and integers $D_n\ge1$ satisfy
\[
        (D_n+1) A(t_n,p_n)=o(t_n).
\]
For each vertex $v$, let $Z_v=Z_v((G_n)_{p_n})$ be the number of cycles of length less than $r$ in $(G_n)_{p_n}$ containing $v$.  Then either
\[
        \limsup_{n\to\infty} h_r^{\mathrm f}(G_n)=\infty,
\]
or, for all sufficiently large $n$,
\[
        \sum_{v\in V(G_n)}
        \mathbb E\binom{Z_v}{D_n+1}
        \ge \frac1{2t_n}.                         \tag{C}
\]
Equivalently, whenever for some $n$ the left side of (C) is less than $1/(2t_n)$, one has
\[
        h_r^{\mathrm f}(G_n)
        \ge
        \frac{t_n}{(D_n+1)A(t_n,p_n)}.
\]
In particular, for $r=5$, every bounded-$h_5^{\mathrm f}$ sequence must have, at every useful retention probability, aggregate rooted $(D_n+1)$st factorial moment for retained $4$-cycles at least of order $1/t_n$.
\end{corollary}

\begin{proof}
If the left side of (C) is less than $1/(2t_n)$, Lemma \ref{lem:factorial-spread} gives
\[
        \Pr\bigl(\lambda_{<r}((G_n)_{p_n})\le D_n\bigr)
        > 1-\frac1{2t_n}
        > \frac1{2t_n}.
\]
Proposition \ref{prop:fractional-spread-criterion} then yields
\[
        h_r^{\mathrm f}(G_n)
        \ge
        \frac{t_n}{(D_n+1)A(t_n,p_n)}.
\]
Since $(D_n+1)A(t_n,p_n)=o(t_n)$, this lower bound tends to infinity along every subsequence on which (C) fails.  Hence bounded $h_r^{\mathrm f}$ forces (C) eventually.
\end{proof}

\begin{remark}[How this frames the endgame]
Theorem \ref{thm:fractional-obstruction-dichotomy} and Corollary \ref{cor:rooted-cluster-obstruction} are deliberately density-free.  They say that the positive families in this paper---clique-organised graphs, projective-plane line graphs, and the Bohman--Keevash tracking-time triangle-free-process graph---succeed because the true fractional-cost obstruction \textup{(O)} fails.  The rooted obstruction \textup{(C)} is a sufficient local certificate for failure of the present spread method, but it is not equivalent to \textup{(O)}.  Burling graphs illustrate the distinction: Corollary~\ref{cor:burling-test-fails} proves that their tight $K_{2,M}$ clusters realize the factorial-moment obstruction \textup{(C)}, while Corollary~\ref{cor:k2m-star-cheap} shows that an isolated $K_{2,M}$ has very small fractional deletion cost.  Thus a genuinely hard Burling sequence would have to realize the stronger obstruction \textup{(O)} by preventing these cheap star deletions from being globalized through the recursion.  Determining whether that happens is a concrete recursive obstruction problem for the Burling construction, not a consequence of Proposition~\ref{prop:burling-factorial-test} alone.
\end{remark}

\subsection{\texorpdfstring{Beyond polynomial density in the integer problem}{Beyond polynomial density in the integer problem}}\label{subsec:quasipoly-density}

The proof of Theorem \ref{thm:poly-sparse} is effective.  The quantitative estimate below records how far the present method extends beyond fixed polynomial density.  The proof is included with an explicit threshold ledger, because the quasi-polynomial corollary depends on distinguishing a quadratic threshold in the edge exponent from a larger iterated threshold.  The key technical choice is that the inner induction terminates by the fixed compact-core theorem rather than by the crude inequality $e\le |V|^2/2$; the latter would square the current density constant at every outer level.

The base exponent used in the quantitative subsection below is deliberately chosen strictly inside the qualitative near-quadratic sparse-core range.  The qualitative base theorem allows a larger endpoint exponent, but the fixed gap in
\[
        A_b=2+\frac1{3r-5}
\]
provides uniform threshold slack in the compact-core, edge-core, and random-thinning estimates.  This choice is only for effectivity and does not change Theorem~\ref{thm:poly-sparse}.

For fixed $r,k$, let $M_{r,k}(P,C)$ be the least threshold, if it exists, such that every graph $G$ with
\[
        \chi(G)\ge M_{r,k}(P,C),
        \qquad
        e(G)\le C\chi(G)^P
\]
contains a subgraph of girth at least $r$ and chromatic number at least $k$.

\begin{lemma}[Uniform slack constants]\label{lem:uniform-slack}
Fix $r\ge4$ and put
\[
        \eta_0=\frac1{16(r-1)},\qquad
        \zeta_0=\frac14,\qquad
        \xi_0=\frac38 .
\]
Let $x>A$ and let $u\ge0$ satisfy
\[
        \sigma_\ell:=\ell A-(\ell-1)x-u\ge\frac12
        \qquad(3\le \ell<r).
\]
Set $\gamma=A-\eta_0$.  Then $\eta_0<A-u$, $\gamma-u\ge\zeta_0$, and, for every $3\le\ell<r$,
\[
        \ell(u-\gamma)+x+(x-u+\eta_0)(\ell-2)\le u-\xi_0 .
\]
\end{lemma}

\begin{proof}
Since $x>A$,
\[
        \sigma_\ell
        =\ell A-(\ell-1)x-u
        = A-u-(\ell-1)(x-A).
\]
The hypothesis therefore implies $A-u\ge 1/2$.  Hence
\[
        \eta_0<A-u,
        \qquad
        \gamma-u=A-u-\eta_0
        \ge \frac12-\frac1{16(r-1)}\ge\frac14.
\]
Moreover,
\begin{align*}
 u-\{\ell(u-\gamma)+x+(x-u+\eta_0)(\ell-2)\}
 &=\ell A-(\ell-1)x-u-(2\ell-2)\eta_0 \\
 &=\sigma_\ell-(2\ell-2)\eta_0 \\
 &\ge \frac12-\frac{2(r-1)}{16(r-1)}
 =\frac38 .
\end{align*}
This proves the three assertions.
\end{proof}

\begin{lemma}[Quantitative input thresholds]\label{lem:quant-input-thresholds}
Fix $r\ge4$ and $k\ge2$, and put $q=k-1$.  Let
\[
        \beta_*:=1+\frac1{4r},
        \qquad
        A_b=2+\frac1{3r-5}.
\]
There are positive constants
\[
        \eps_{\rm comp},\ \eps_{\rm edge},\ \eps_{\rm rand}>0
\]
which depend only on $r$, and constants
\[
        K_{\rm comp},\ K_{\rm edge},\ K_{\rm rand}>1
\]
which depend only on $r,k$, such that the following three assertions hold.  Put
\[
        \eps_*=
        \min\{\eps_{\rm comp},\eps_{\rm edge},\eps_{\rm rand}\},
        \qquad
        K_*=
        \max\{K_{\rm comp},K_{\rm edge},K_{\rm rand}\}.
\]
Then the following uniform threshold implications hold.
\begin{enumerate}[label=\textup{(\roman*)}]
\item If $\chi(G)=m$ and $|V(G)|\le Bm^{\beta_*}$, then $h_r(G)\ge k$ whenever
\[
        m^{\eps_*}\ge K_*(B+2)^{K_*}.        \tag{QI1}
\]
\item If $\chi(G)=m$ and $e(G)\le Bm^{A_b}$, then $h_r(G)\ge k$ whenever
\[
        m^{\eps_*}\ge K_*(B+2)^{K_*}.        \tag{QI2}
\]
\item In the random-thinning branch of Proposition~\ref{prop:mixed}, assume the endpoint slack
\[
        \min_{3\le \ell<r}\{\ell A-(\ell-1)x-u\}\ge \frac12
\]
for the current edge exponent $x$, vertex exponent $u$, and previous exponent level $A$.  If the current edge and vertex constants are both at most $B$, then the Chernoff union bound over all $q$-colourings, the condition $p\le1$, and the short-cycle Markov estimate used in that branch all hold whenever
\[
        m^{\eps_*}\ge K_*(B+2)^{K_*}.        \tag{QI3}
\]
\end{enumerate}
\end{lemma}

\begin{proof}
We give the parameter ledger explicitly.  Constants denoted $C_i$ may depend on $r$ and $k$, but not on $B$, $m$, or the graph.  Throughout the proof $q=k-1$ is fixed.

\smallskip
\noindent\emph{Compact-core input.}
Choose a sampling exponent $\alpha_c$ strictly inside the interval allowed in the proof of Theorem~\ref{thm:compact-core} with $\beta=\beta_*$.  Equivalently,
\[
        \delta_{c,0}:=2-\alpha_c-\beta_*>0,
        \qquad
        \delta_{c,\ell}:=2-\alpha_c-\beta_*\ell+\alpha_c\ell>0
        \quad(3\le\ell<r).
\]
Let
\[
        \delta_c=\min\{\delta_{c,0},\delta_{c,3},\ldots,\delta_{c,r-1}\}>0.
\]
With $p=m^{-\alpha_c}$ and $m\ge 2q$, the defect bound gives
\[
        p\mu_q(G)\ge \frac1{8q}m^{2-\alpha_c}.        \tag{1}
\]
The number of $q$-colourings is at most $q^{|V(G)|}\le \exp(Bm^{\beta_*}\log q)$, so the preservation part of Theorem~\ref{thm:random-extraction} follows from
\[
        m^{\delta_{c,0}}\ge C_1(B+2)(q\log q+2).        \tag{2}
\]
For $3\le\ell<r$ we use the trivial estimate
\[
        C_\ell(G)\le |V(G)|^\ell\le B^\ell m^{\beta_*\ell}.
\]
Thus
\[
        p^\ell C_\ell(G)
        \le B^\ell m^{\beta_*\ell-\alpha_c\ell}
        =B^\ell m^{2-\alpha_c-\delta_{c,\ell}},
\]
and the short-cycle deletion part is implied by
\[
        m^{\delta_c}\ge C_2(B+2)^{r-1}(q+2).          \tag{3}
\]
Equations (2) and (3) are consequences of (QI1), for example with
\[
        \eps_{\rm comp}=\delta_c/2
\]
and $K_{\rm comp}$ chosen large enough to dominate the fixed powers of $q+2$ and the constants $C_1,C_2$.

\smallskip
\noindent\emph{Near-quadratic sparse-core input.}
Write
\[
        A_b=2+\eps_b,
        \qquad
        \eps_b=\frac1{3r-5}.
\]
This lies a fixed positive distance below the endpoint $2+2/(3r-5)$ in Theorem~\ref{thm:edge-core}.  In the proof of that theorem one first passes to an $m$-critical subgraph, so
\[
        |V(G)|\le \frac{2e(G)}{m-1}\le 4B m^{1+\eps_b}
\]
for all large $m$.  Choosing the edge-retention exponent strictly inside the allowed interval gives a fixed gap $\delta_e=\delta_e(r)>0$.  The preservation inequality has the form
\[
        m^{\delta_e}\ge C_3(B+2)(q\log q+2),          \tag{4}
\]
and the spectral cycle estimate
\[
        C_\ell(G)\le (2e(G))^{\ell/2}
\]
gives, for each $3\le\ell<r$, a short-cycle condition implied by
\[
        m^{\delta_e}\ge C_4(B+2)^r(q+2)^{C_5}.        \tag{5}
\]
Thus (QI2) holds with $\eps_{\rm edge}=\delta_e/2$ and $K_{\rm edge}$ sufficiently large.  Notice that this proves the base uniformly for every edge exponent at most $A_b$: if $e(G)\le Bm^x$ with $x\le A_b$, then also $e(G)\le Bm^{A_b}$.

\smallskip
\noindent\emph{Random branch inside the mixed step.}
Use the notation of Proposition~\ref{prop:mixed}.  At the current stage the graph has at most $Bm^u$ vertices, at most $Bm^x$ edges, maximum degree at most $Bm^{x-u+\eta}$ after the high-degree pruning, and every $q$-colouring has more than $m^\gamma$ monochromatic edges.  Set
\[
        p=Lm^{u-\gamma}.
\]
Choose
\[
        L\ge 32(B+1)\log q+32.                       \tag{6}
\]
Then for each fixed $q$-colouring the expected number of retained monochromatic edges is at least $Lm^u$, and Chernoff gives failure probability at most $\exp(-Lm^u/8)$.  Since the graph has at most $Bm^u$ vertices, there are at most $q^{Bm^u}$ colourings, and (6) makes the union bound valid.

By Lemma~\ref{lem:uniform-slack}, with the fixed choice $\eta=\eta_0$ and $\gamma=A-\eta_0$, the endpoint slack in (QI3) gives
\[
        \gamma-u\ge \zeta_0                                      \tag{7}
\]
and, for every $3\le\ell<r$,
\[
        \ell(u-\gamma)+x+(x-u+\eta_0)(\ell-2)
        \le u-\xi_0.                                             \tag{8}
\]
Here (8) is exactly the exponent of
\[
        p^\ell\,e(G)\,\Delta(G)^{\ell-2}
\]
after substituting $p=Lm^{u-\gamma}$ and $\Delta(G)\le Bm^{x-u+\eta_0}$.  The constants must also be tracked: for each $3\le\ell<r$, the maximum-degree cycle count gives
\[
        \mathbb E C_\ell(G_p)
        \le 2 B^{\ell-1}L^\ell
        m^{\ell(u-\gamma)+x+(x-u+\eta)(\ell-2)}.
\]
Condition $p\le1$ follows from
\[
        m^{\zeta_0}\ge L.                                      \tag{9}
\]
By (8), and since $B,L\ge1$ and $\ell\le r-1$, summing over $3\le\ell<r$ gives
\[
        \mathbb E C_{<r}(G_p)\le C_6(r) B^{r-2} L^{r-1}m^{u-\xi_0}.       \tag{10}
\]
Thus Markov's inequality gives $C_{<r}(G_p)<Lm^u/4$ with probability at least $3/4$ whenever
\[
        m^{\xi_0}\ge 4C_6(r) B^{r-2}L^{r-2}.                            \tag{11}
\]
Since (6) allows $L\le C_8(B+2)(q\log q+2)$, the requirements (6), (9), and (11) are all implied by
\[
        m^{\min\{\zeta_0,\xi_0\}/2}
        \ge C_9(B+2)^{2r-4}(q\log q+2)^{r-2}.
\]
Thus (QI3) holds with
\[
        \eps_{\rm rand}=\frac12\min\{\zeta_0,\xi_0\}
\]
and $K_{\rm rand}$ sufficiently large.  Taking the minimum of the three exponent gaps and the maximum of the three threshold constants gives the displayed $\eps_*$ and $K_*$.  This completes the proof.
\end{proof}

\begin{lemma}[Quantitative bookkeeping for one exponent step]\label{lem:quant-bookkeeping}
Fix $r\ge4$ and $k\ge2$.  Let $q=k-1$, let
\[
        \beta_*:=1+\frac1{4r},
        \qquad
        A_b=2+\frac1{3r-5},
        \qquad
        \Delta=\frac1{2(r-1)},
\]
and set $A_j=A_b+j\Delta$.  Suppose that, for some $j\ge0$, every graph with $\chi(G)=m$, $e(G)\le C m^x$ and $x\le A_j$ has $h_r(G)\ge k$ once
\[
        m\ge \exp T_j(C).
\]
Then there is a constant $C_0=C_0(r,k)$ such that the same assertion for all $x\le A_{j+1}$ holds with a logarithmic threshold $T_{j+1}$ satisfying
\[
        T_{j+1}(C)
        \le
        \max\left\{
        C_0\bigl((j+1)^2+1+\log(C\vee2)\bigr),
        T_j\bigl((2q+2)^{C_0(j+1)}\bigr)
        \right\}.
\]
\end{lemma}

\begin{proof}
Only the range $A_j<x\le A_{j+1}$ is new.  Pass to an $m$-critical subgraph.  Then
\[
        |V(G)|\le \frac{2e(G)}{m-1}\le 2C m^{x-1}
\]
for all large $m$.  Apply the mixed-peeling proof, Proposition~\ref{prop:mixed}, with previous exponent $A=A_j$, edge exponent $x$, vertex exponent $u=x-1$, and initial constant $2C$.  Because $x-A_j\le\Delta$, the endpoint slack is at least
\[
        1-(r-1)(x-A_j)\ge \frac12.
\]
Use the fixed constant $\eta_0=1/(16(r-1))$ from Lemma~\ref{lem:uniform-slack}, and set
\[
        \eta=\eta_0,
        \qquad
        \gamma=A_j-\eta_0,
        \qquad
        d(u)=x-u+\eta_0.
\]
The endpoint slack hypothesis needed for Lemma~\ref{lem:uniform-slack} holds at the start of the outer step and at each later inner-descent stage, because the edge exponent $x$ is fixed while the current vertex exponent only decreases.  Thus the inequalities required in the random branch of Proposition~\ref{prop:mixed} hold uniformly with the constants $\zeta_0,\xi_0$ from Lemma~\ref{lem:uniform-slack}.  With this fixed choice, one high-degree descent lowers the current vertex exponent by exactly $\eta_0$.  If $N_0(j)$ denotes the maximum possible number of such descents before the vertex exponent reaches the compact-core threshold $\beta_*$, then
\[
        N_0(j)\le
        \left\lceil \frac{A_{j+1}-1-\beta_*}{\eta_0}\right\rceil_+
        \le C_r(j+1),
\]
where $\lceil y\rceil_+=\max\{0,\lceil y\rceil\}$.

During a high-degree descent, the current subgraph has chromatic number at least half of the previous one.  Rewriting the edge and vertex bounds in the new chromatic parameter multiplies the current constants by at most
\[
        \max\{2^x,2^{u'+1}\}\le \exp\{C_1(r)(j+1)\},
\]
since all relevant exponents are $O_r(j+1)$.  More explicitly, after $s$ high-degree descents in this outer increment, let $t_s$ be the current chromatic number, $C_s$ the current density constant, and $u_s$ the current vertex exponent.  Then
\[
        t_s\ge 2^{-s}m,
        \qquad
        C_s\le C\exp\{C_1(r)(j+1)s\},
        \qquad
        u_s=x-1-s\eta_0.
\]
The descent lowers the vertex exponent by $\eta_0$, and hence at most $C_2(r)(j+1)$ such descents occur before the fixed compact exponent $\beta_*$ is reached.  Therefore every constant appearing before the terminal compact-core step is at most
\[
        C\exp\{C_3(r)(j+1)^2\}.                     \tag{5}
\]
The terminal compact-core requirement and the large-defect random-thinning requirement are then covered by Lemma~\ref{lem:quant-input-thresholds}.  By (5), these branches require only
\[
        \log m\ge C_4(r,k)\bigl((j+1)^2+1+\log(C\vee2)\bigr).      \tag{6}
\]

In the small-defect branch, the graph of monochromatic edges has exponent $\gamma=A_j-\eta<A_j$ and chromatic number at least a factor $1/(2q)$ of the current chromatic number.  If this chromatic number is $s_\psi$, then $s_\psi\ge t'/(2q)$ and, after the harmless losses already accounted for in the current parameter,
\[
        e(M_\psi)\le (t')^\gamma\le (2q)^\gamma s_\psi^\gamma
        \le (2q+2)^{C_5(r)(j+1)}s_\psi^\gamma .
\]
Thus the density constant sent to the previous exponent level is independent of the original constant $C$ and is at most $(2q+2)^{C_5(r)(j+1)}$.

There is also a harmless loss in the chromatic parameter itself, and we record it explicitly.  All inequalities above are measured in the current chromatic parameter.  If the current graph has chromatic number $t'$ and it arose after $s$ high-degree descents from the original graph of chromatic number $m$, then $t'\ge 2^{-s}m$.  A threshold requirement $t'\ge \exp U$ is therefore implied by
\[
        \log m\ge U+s\log2.
\]
After the small-defect branch the monochromatic-edge graph has chromatic number at least $1/(2q)$ times the current chromatic number.  Since $s\le N_0(j)\le C_r(j+1)$ during one outer increment, the total multiplicative loss before applying the previous level is at most $(2q)2^{C_6(r)(j+1)}$.  Hence it is enough to require
\[
        \log m\ge T_j((2q+2)^{C_5(r)(j+1)})+C_7(r)(j+1)+\log(2q+2).
\]
After increasing $C_0$, the additive term $C_7(r)(j+1)+\log(2q+2)$ is absorbed by $C_0((j+1)^2+1+\log(C\vee2))$.  Taking $C_0$ larger than the constants in (6) and in this paragraph gives the stated recurrence.
\end{proof}

\begin{theorem}[Quantitative sparse threshold]\label{thm:quantitative-threshold}
Fix $r\ge4$ and $k\ge2$.  There is a constant $B>0$, depending only on $r$ and $k$, such that for every $P\ge2$ and $C\ge2$,
\[
        M_{r,k}(P,C)
        \le
        \exp\bigl(B(P+\log C)^2\bigr).
\]
\end{theorem}

\begin{proof}
Let $A_b=2+1/(3r-5)$ and $\Delta=1/(2(r-1))$, and put $A_j=A_b+j\Delta$.  For each $j\ge0$ we choose $T_j:[2,\infty)\to[0,\infty)$ to be a nondecreasing majorant of all logarithmic thresholds certified up to edge exponent $A_j$ and density constant $C$.  Thus, if a statement is valid above $\exp U(C)$, we are free to replace $U$ by any larger nondecreasing function of $C$.  All threshold functions below are understood after this monotone-majorant normalization.  Lemma~\ref{lem:quant-input-thresholds}(ii) gives
\[
        T_0(C)\le B_0(1+\log(C\vee2))                 \tag{7}
\]
for a constant $B_0=B_0(r,k)$.  Lemma~\ref{lem:quant-bookkeeping} gives, for all $j\ge0$,
\[
        T_{j+1}(C)
        \le
        \max\left\{
        C_0((j+1)^2+1+\log(C\vee2)),
        T_j((2q+2)^{C_0(j+1)})
        \right\}.                                      \tag{8}
\]
We now solve (8).  Appendix~\ref{app:threshold-ledger} records the same bookkeeping in a single consolidated ledger for the reader who wants every threshold dependency in one place.  Put $K_0=C_0\log(2q+2)$ and choose $A_*=1+K_0+C_0$.  Taking $B$ sufficiently large in terms of $B_0,C_0,A_*$, we prove by induction that
\[
        T_j(C)\le B\bigl(A_*(j+1)^2+\log(C\vee2)\bigr)       \tag{9}
\]
for all $j\ge0$.  The case $j=0$ follows from (7).  If (9) holds at level $j$, then the first term in (8) is dominated by the right side of (9) with $j+1$ in place of $j$.  For the second term,
\[
        T_j((2q+2)^{C_0(j+1)})
        \le B\bigl(A_*(j+1)^2+K_0(j+1)\bigr).
\]
Since
\[
        A_*\bigl((j+2)^2-(j+1)^2\bigr)=A_*(2j+3)\ge K_0(j+1),
\]
this is also bounded by the right side of (9) at level $j+1$.  Thus (9) holds for all $j$.

Choose $j$ minimal with $P\le A_j$.  Then $j=O_r(P)$.  Estimate (9) gives
\[
        \log M_{r,k}(P,C)
        \le O_{r,k}(P^2+\log C)
        \le O_{r,k}((P+\log C)^2),
\]
for $P,C\ge2$.  This proves the theorem.
\end{proof}

\begin{corollary}[An unconditional quasi-polynomial density range]\label{cor:quasipoly-density}
Fix $r\ge4$ and $k\ge2$.  For every $1<a<3/2$ and every $C_0>0$, every graph $G$ of sufficiently large chromatic number satisfying
\[
        e(G)\le \exp\bigl(C_0(\log\chi(G))^a\bigr)
\]
contains a subgraph $H$ with $\girth(H)\ge r$ and $\chi(H)\ge k$.
\end{corollary}

\begin{proof}
Let $m=\chi(G)$ and set
\[
        P_m=\max\{2,C_0(\log m)^{a-1}\}.
\]
For all large $m$, $e(G)\le m^{P_m}\le2m^{P_m}$.  Theorem~\ref{thm:quantitative-threshold}, applied with constant $2$, gives
\[
        \log M_{r,k}(P_m,2)
        \le B(P_m+\log2)^2
        =O((\log m)^{2(a-1)}).
\]
Since $2(a-1)<1$, this is $o(\log m)$.  Hence $m\ge M_{r,k}(P_m,2)$ for all sufficiently large $m$, and Theorem~\ref{thm:poly-sparse} applies.
\end{proof}

\begin{remark}[Why the compact base is necessary]\label{rem:no-squaring}
If the inner induction were terminated by $e\le |V|^2/2$, the current constant $C'$ would be replaced by roughly $(C')^2$.  Iterating over $\Theta(P)$ outer levels would double the coefficient of $\log C$ at each level and would destroy the quadratic threshold estimate.  The fixed compact-core base avoids this squaring and keeps the dependence on the current constant logarithmic at each level.
\end{remark}

\begin{remark}[Why imbalanced splits do not help this proof]\label{rem:imbalanced-split-fails}
A natural attempt to improve the quadratic threshold is to replace the balanced chromatic split by an imbalanced split.  At outer level $j$, suppose the high-degree branch is used only when it has chromatic number at least $(1-\rho_j)t$.  This makes the high-degree branch cheaper, but the complementary branch has chromatic number only $\rho_jt$.  In the small-defect branch, the monochromatic-edge graph then has density constant, in its own chromatic parameter, of order
\[
        (q/\rho_j)^{\Theta_r(j)}.
\]
Thus the level-$j$ threshold must increase by at least $\Omega_r(j\log(1/\rho_j))$ from the small-defect branch.  If $\rho_j$ stays bounded away from zero, the high-degree branch already gives a quadratic cumulative loss; if $\rho_j\to0$, the small-defect increments still sum to a quadratic scale, and the choice $\rho_j=1/(j+3)$ is worse, giving $\Omega_r(P^2\log P)$.  Therefore this one-parameter split scheme cannot yield $\log M_{r,k}(P,C)=o(P^2)$.
\end{remark}

\begin{remark}[What would be needed to reach density exponent $a<2$]\label{rem:a-less-two-target}
The preceding theorem gives $\log M_{r,k}(P,C)=O_{r,k}((P+\log C)^2)$ and hence the range $1<a<3/2$.  A near-linear threshold $O_{r,k}(P\operatorname{polylog}P+\log(C\vee2))$ would extend the same argument to every $1<a<2$, by taking $P=C_0(\log m)^{a-1}$.  The previous remark explains why such a threshold requires a different peeling or defect-reduction step, not only a different choice of chromatic split.
\end{remark}

\begin{remark}[The quantitative barrier left by Section \ref{sec:poly}]
Corollary \ref{cor:quasipoly-density} reaches $1<a<3/2$ because Theorem \ref{thm:quantitative-threshold} gives a quadratic logarithmic threshold in $P$.  Reaching $a<2$ remains a concrete quantitative target for any future refinement of the peeling/thinning bootstrap.
\end{remark}

\subsection{\texorpdfstring{Sparse-regime bounds as functions of $k$}{Sparse-regime bounds as functions of k}}

For fixed $r,P,C$, define the sparse-regime threshold
\[
        f_{P,C}(k,r)=\min\{M:\ e(G)\le C\chi(G)^P,\ \chi(G)\ge M
        \Rightarrow h_r(G)\ge k\}.
\]
The next proposition records the effective $k$-dependence furnished by the proof.  It is independent of the quasi-polynomial-density discussion above: here $P$ and $C$ are fixed and only $k$ varies.

\begin{lemma}[Polynomial $q$-dependence of one mixed step]
\label{lem:poly-q-step}
Fix $r\ge4$ and fix an upper edge exponent $P$.  Suppose that, at one edge-exponent level, all threshold requirements have the form
\[
        m\ge (D+2)^a(q+2)^b
\]
for some exponents $a,b$ depending only on $r$ and $P$, where $D$ is the current density constant and $q=k-1$.  Then one mixed-peeling increment preserves this polynomial form: the next level has a sufficient threshold
\[
        m\ge (D+2)^{a'}(q+2)^{b'}
\]
for exponents $a',b'$ depending only on $r,P,a,b$.
\end{lemma}

\begin{proof}
The compact-core terminal branch and the large-defect random-thinning branch have fixed positive exponent gaps when $r$ and $P$ are fixed.  In both branches the quantitative input inequalities are polynomial in the current density constant and in $q\log q+2$, so they are implied by a threshold of the displayed polynomial form after increasing $a'$ and $b'$.

A high-degree descent stays at the current edge-exponent level and lowers only the vertex exponent.  For fixed $r$ and $P$ the number of such descents in one increment is bounded by a constant depending only on $r$ and $P$.  Each descent rewrites an edge or vertex bound in a subgraph whose chromatic number is at least a fixed fraction of the previous one; this multiplies the current density constant by a fixed power of $D+2$ and $q+2$.  Therefore all density constants produced before the terminal branch are bounded by $(D+2)^E(q+2)^F$ for constants $E,F$ depending only on $r$ and $P$.

Only the small-defect branch invokes the previous edge-exponent level.  In that branch the monochromatic-edge graph has chromatic number at least a $1/(2q)$-fraction of the current chromatic number.  Its density constant, measured in its own chromatic parameter, is again bounded by $(D+2)^E(q+2)^F$ after increasing $E,F$.  Applying the previous-level hypothesis to this constant gives
\[
        \bigl((D+2)^E(q+2)^F+2\bigr)^a(q+2)^b
        \le (D+2)^{a'}(q+2)^{b'}
\]
for suitable $a',b'$.  This proves the lemma.
\end{proof}

\begin{proposition}[Polynomial sparse thresholds in $k$]\label{prop:sparse-threshold-k}
Fix $r\ge4$, $P\ge2$ and $C\ge2$.  There is a constant $K=K(r,P,C)$ such that for every $k\ge2$,
\[
        f_{P,C}(k,r)\le k^K.
\]
In particular,
\[
        f_{P,C}(k,5)\le k^{O_{P,C}(1)}.
\]
\end{proposition}

\begin{proof}
The case $k=2$ is immediate, so assume $q=k-1\ge2$.  We track only the dependence on $q$, since $r,P,C$ are fixed throughout the proof.  Put
\[
        A_b=2+\frac1{3r-5},
        \qquad
        \Delta=\frac1{2(r-1)},
\]
and choose $J=J(r,P)$ such that $P\le A_b+J\Delta$.
For $0\le j\le J$ let $\mathcal T_j(q,D)$ be a threshold sufficient for all edge exponents at most $A_b+j\Delta$ and all edge-density constants at most $D$.  We prove by induction on $j$ that there are exponents $a_j,b_j$, depending on $r,P,C$ but not on $q$ or $D$, such that
\[
        \mathcal T_j(q,D)
        \le (D+2)^{a_j}(q+2)^{b_j}
        \qquad(D\ge1).                         \tag{1}
\]

For $j=0$, the compact-core and near-quadratic sparse-core bases have fixed positive exponent gaps.  The defect estimate
\[
        \mu_q(G)\ge \frac{(m-q)^2}{2q}\ge \frac{m^2}{8q}
        \qquad(m\ge2q)
\]
shows that the preservation condition in Theorem~\ref{thm:random-extraction} is implied by an inequality of the form
\[
        m^{\eta_1}\ge C_1(D+2)^{A_1}(q\log q+2)^{B_1},
\]
where $\eta_1>0$ and $A_1,B_1,C_1$ depend only on the fixed exponent gaps and on $r$.  The short-cycle inequalities in the compact base use $C_\ell(G)\le |V(G)|^\ell$, while the sparse base uses $C_\ell(G)\le (2e(G))^{\ell/2}$.  Since all exponent gaps are fixed, these inequalities have the same polynomial form.  Enlarging the exponents gives (1) for $j=0$.

Assume (1) holds at level $j$.  Lemma~\ref{lem:poly-q-step} gives the abstract closure mechanism; the next paragraphs spell out the recurrence explicitly.  In one mixed-peeling increment all exponent gaps are bounded below in terms of $r$ and $P$.  Lemma~\ref{lem:quant-input-thresholds} therefore gives constants
\[
        A,B,E,F>0,
\]
depending only on $r$ and $P$, with the following two consequences.  First, the terminal compact-core branch and the large-defect random-thinning branch are guaranteed once
\[
        m\ge (D+2)^A(q+2)^B.                         \tag{2}
\]
Second, a high-degree descent does not pass to the previous edge-exponent level.  It stays at the current edge exponent and lowers the vertex exponent by the fixed amount supplied by the mixed-peeling step.  Since, for fixed $r$ and $P$, the number of such descents in one exponent increment is bounded by a constant depending only on $r$ and $P$, the cumulative loss in the density constant over all high-degree descents is polynomial in the current density constant and in $q+2$.  Thus, after all possible high-degree descents inside the current exponent increment, the density constants are bounded by
\[
        (D+2)^E(q+2)^F.                         \tag{3}
\]

Only in the small-defect branch do we invoke the previous edge-exponent level.  In that branch the monochromatic-edge graph has chromatic number at least a $1/(2q)$-fraction of the current chromatic number and has edge exponent strictly below the previous level.  Measured in its own chromatic parameter, its density constant is also bounded by the quantity in (3), after increasing $E,F$ if necessary.  The retention multiplier in the random branch is polynomial in the current density constant and in $q\log q$, and this is absorbed into the same polynomial form.

Thus we have the explicit recurrence
\[
        \mathcal T_{j+1}(q,D)
        \le
        \max\left\{(D+2)^A(q+2)^B,
        \mathcal T_j\bigl(q,(D+2)^E(q+2)^F\bigr)\right\}.        \tag{4}
\]
Applying the induction hypothesis to the second term gives
\[
\begin{aligned}
        \mathcal T_j\bigl(q,(D+2)^E(q+2)^F\bigr)
        &\le \bigl((D+2)^E(q+2)^F+2\bigr)^{a_j}(q+2)^{b_j}  \\
        &\le (D+2)^{2Ea_j}(q+2)^{2Fa_j+b_j},
\end{aligned}
\]
where the harmless factor $2$ is absorbed because $D+2,q+2\ge2$.  Therefore (4) is bounded by
\[
        (D+2)^{a_{j+1}}(q+2)^{b_{j+1}}
\]
if, for example,
\[
        a_{j+1}=\max\{A,2Ea_j\},
        \qquad
        b_{j+1}=\max\{B,2Fa_j+b_j\}.
\]
This proves (1) at level $j+1$.

Taking $D=C$ and $j=J$ gives
\[
        f_{P,C}(k,r)
        \le (C+2)^{a_J}(q+2)^{b_J}
        \le k^{K(r,P,C)}
\]
after increasing $K(r,P,C)$.  This proves the proposition.
\end{proof}

Thus the sparse regime has a dramatically smaller upper bound in $k$ than the currently known unrestricted $r=5$ lower-bound examples.  Those Burling-based examples necessarily lie outside every fixed polynomial-density critical-core regime, whereas every fixed polynomial-density regime has polynomial thresholds in the target chromatic number.

\appendix
\section{Explicit threshold ledger for the quantitative sparse theorem}\label{app:threshold-ledger}

This appendix records the quantitative dependence used in Theorem~\ref{thm:quantitative-threshold}.  It is included to make explicit that the constants accumulated in the mixed peeling proof grow quadratically in the edge exponent.  Throughout this appendix, all constants may depend on the fixed parameters $r$ and $k$, but not on $P$, $C$, or the chromatic number $m$.

Let
\[
        A_b=2+\frac1{3r-5},\qquad
        \Delta=\frac1{2(r-1)},\qquad
        A_j=A_b+j\Delta .
\]
For $C\ge2$, let $N_j(C)$ be a monotone majorant for the least threshold with the following property: every graph $G$ with chromatic number at least $N_j(C)$ and
\[
        e(G)\le C\chi(G)^x,
        \qquad x\le A_j,
\]
contains a subgraph of girth at least $r$ and chromatic number at least $k$.  We prove, by induction on $j$, that there are constants $B_1,B_2$ such that
\[
        \log N_j(C)\le B_1(j+1)^2+B_2\log C .      \tag{A.1}
\]

The base case $j=0$ follows from the near-quadratic sparse-core theorem.  Indeed, $A_b=2+1/(3r-5)$ lies a fixed positive distance below the endpoint $2+2/(3r-5)$ in Theorem~\ref{thm:edge-core}.  Hence the edge-core proof has a fixed exponent margin.  Equivalently, there are constants $b_0,b_1$ such that
\[
        \log N_0(C)\le b_0+b_1\log C.              \tag{A.2}
\]
This is the only place where the fixed interior choice of $A_b$ is used.

Assume now that (A.1) is known at level $j$.  We prove the level $j+1$ bound.  The only new range is $A_j<x\le A_{j+1}$.  Passing to an $m$-critical subgraph gives
\[
        |V(G)|\le 2C m^{x-1}
\]
for all large $m$.  The mixed peeling argument is run with previous edge level $A_j$, edge exponent $x$, and initial vertex exponent $u_0=x-1$.  Since $x-A_j\le \Delta$, the endpoint slack is at least $1/2$.  With the fixed value $\eta_0=1/(16(r-1))$ from Lemma~\ref{lem:uniform-slack}, every high-degree descent lowers the vertex exponent by $\eta_0$.

The inner descent stops either when the vertex exponent reaches the fixed compact exponent $\beta_*=1+1/(4r)$, or when the small-defect branch appears.  Thus the number of high-degree descents in one outer increment is at most
\[
        R_j\le b_2(j+1).                            \tag{A.3}
\]
At the $s$th inner stage, measured in the current chromatic parameter $t_s$, the edge and vertex bounds have constants at most
\[
        C_s\le C\exp\{b_3(j+1)s\}.                  \tag{A.4}
\]
This is because rewriting a bound in a subgraph whose chromatic number is at least half of the previous one multiplies a constant by at most $2^{O_r(j+1)}$.  Combining (A.3) and (A.4), every constant appearing before the terminal compact-core branch is at most
\[
        C\exp\{b_4(j+1)^2\}.                        \tag{A.5}
\]
The compact-core terminal branch and the large-defect random-thinning branch have fixed positive exponent margins; by Lemma~\ref{lem:quant-input-thresholds}, they are therefore available once
\[
        \log m\ge b_5\bigl((j+1)^2+1+\log C\bigr).  \tag{A.6}
\]
Here (A.5) has been substituted into the logarithmic dependence on the current density constant.

It remains to record the small-defect branch.  Suppose it occurs at current chromatic number $t'$.  The graph of monochromatic edges has chromatic number $s_\psi\ge t'/(2q)$ and at most $(t')^\gamma$ edges, where $\gamma=A_j-\eta_0<A_j$.  Measured in its own chromatic parameter,
\[
        e(M_\psi)\le (t')^\gamma
        \le (2q)^\gamma s_\psi^\gamma
        \le (2q+2)^{b_6(j+1)}s_\psi^\gamma .        \tag{A.7}
\]
The chromatic-number losses before invoking the previous level are also explicit.  If $s$ high-degree descents preceded the small-defect branch, then $t'\ge 2^{-s}m$, and the branch itself loses a further factor $2q$.  Hence a previous-level threshold $s_\psi\ge N_j(D)$ is guaranteed by
\[
        \log m\ge \log N_j(D)+b_7(j+1)+\log(2q+2),  \tag{A.8}
\]
where $D=(2q+2)^{b_6(j+1)}$.  By the induction hypothesis,
\[
        \log N_j(D)
        \le B_1(j+1)^2+B_2 b_6(j+1)\log(2q+2).     \tag{A.9}
\]
The right side is $O_{r,k}((j+1)^2)$.  Combining (A.6)--(A.9), and increasing the constants in (A.1), gives
\[
        \log N_{j+1}(C)
        \le B_1(j+2)^2+B_2\log C .
\]
This proves (A.1) at level $j+1$.  Since the least $j$ with $P\le A_j$ is $O_r(P)$, Theorem~\ref{thm:quantitative-threshold} follows.

\section*{Acknowledgements}
The author acknowledges the use of OpenAI's ChatGPT during the preparation of this manuscript. While it was used for ideation, formulation, proof exploration and refinement, narrowing the search space, programming, LaTeX formatting and other forms of orchestration, the author nonetheless takes full responsibility for the accuracy of the final contents of this paper.


\begin{thebibliography}{99}

\bibitem{BerkowitzDevlinLeeReichardTownley2018}
R. Berkowitz, P. Devlin, C. Lee, H. Reichard and D. Townley,
\newblock Expected chromatic number of random subgraphs,
\newblock arXiv:1811.02018, 2018.

\bibitem{BohmanKeevash2013}
T. Bohman and P. Keevash,
\newblock Dynamic concentration of the triangle-free process,
\newblock Random Structures Algorithms 58 (2021), 221--293,
\newblock doi:10.1002/rsa.20973.

\bibitem{BKN2023}
B. Bukh, M. Krivelevich and B. Narayanan,
\newblock Colouring random subgraphs,
\newblock Combin. Probab. Comput. 34 (2025), no. 4, 585--595,
\newblock doi:10.1017/S0963548325000069.

\bibitem{DiwanKenkreVishwanathan2008}
A. A. Diwan, S. Kenkre and S. Vishwanathan,
\newblock Circumference, chromatic number and online coloring,
\newblock arXiv:0809.1710, 2008.

\bibitem{ErdosHajnal1966}
P. Erd\H{o}s and A. Hajnal,
\newblock On chromatic number of graphs and set-systems,
\newblock Acta Math. Acad. Sci. Hungar. 17 (1966), 61--99.

\bibitem{Erdos1959}
P. Erd\H{o}s,
\newblock Graph theory and probability,
\newblock Canad. J. Math. 11 (1959), 34--38.

\bibitem{FelsnerJoretMicekTrotterWiechert2018}
S. Felsner, G. Joret, P. Micek, W. T. Trotter and V. Wiechert,
\newblock Burling graphs, chromatic number, and orthogonal tree-decompositions,
\newblock Electron. J. Combin. 25 (2018), Paper No. 1.35.


\bibitem{FGM2013}
G. Fiz Pontiveros, S. Griffiths and R. Morris,
\newblock The triangle-free process and the Ramsey number $R(3,k)$,
\newblock Mem. Amer. Math. Soc. 263 (2020), no. 1274, 125 pp.

\bibitem{Kim1995}
J. H. Kim,
\newblock The Ramsey number $R(3,t)$ has order of magnitude $t^2/\log t$,
\newblock Random Structures Algorithms 7 (1995), 173--207.

\bibitem{KenkreVishwanathan2007}
S. Kenkre and S. Vishwanathan,
\newblock A bound on the chromatic number using the longest odd cycle length,
\newblock J. Graph Theory 54 (2007), 267--276.

\bibitem{MoharWuRandomFractional}
B. Mohar and H. Wu,
\newblock Fractional chromatic number of a random subgraph,
\newblock J. Graph Theory 95 (2020), 467--472,
\newblock doi:10.1002/jgt.22571.

\bibitem{MoharWuTriangleFreeFractional}
B. Mohar and H. Wu,
\newblock Triangle-free subgraphs with large fractional chromatic number,
\newblock Combin. Probab. Comput. 31 (2022), 136--143.

\bibitem{MoharWuKneser}
B. Mohar and H. Wu,
\newblock Subgraphs of Kneser graphs with large girth and large chromatic number,
\newblock Art Discrete Appl. Math. 6 (2023), no. 2, Paper No. 2.11, 7 pp.,
\newblock doi:10.26493/2590-9770.1491.14a.

\bibitem{PTW2026}
S. Pettie, G. Tardos and B. Walczak,
\newblock On a clique game and the Erd\H{o}s--Hajnal problem on high-chromatic high-girth subgraphs,
\newblock Proceedings of the 2026 Annual ACM-SIAM Symposium on Discrete Algorithms (SODA), 2026, 2903--2927,
\newblock doi:10.1137/1.9781611978971.108.

\bibitem{Walczak2015}
B. Walczak,
\newblock Triangle-free geometric intersection graphs with no large independent sets,
\newblock Discrete Comput. Geom. 53 (2015), 221--225,
\newblock doi:10.1007/s00454-014-9645-y.

\bibitem{Rodl1977}
V. R\"odl,
\newblock On the chromatic number of subgraphs of a given graph,
\newblock Proc. Amer. Math. Soc. 64 (1977), 370--371.

\bibitem{Shearer1983}
J. B. Shearer,
\newblock A note on the independence number of triangle-free graphs,
\newblock Discrete Math. 46 (1983), 83--87.


\bibitem{Shinkar2016}
I. Shinkar,
\newblock On coloring random subgraphs of a fixed graph,
\newblock arXiv:1612.04319, 2016.

\end{thebibliography}
\end{document}